\documentclass[a4paper,10pt]{amsart}

\title{Operad profiles of Nijenhuis structures}
\author{Henrik Strohmayer}
\address{Department of Mathematics \\ Stockholm University \\ 106 91 Stockholm \\ Sweden}
\email{henriks@math.su.se}

\usepackage{amsmath,amssymb,amsthm,amscd,latexsym,stmaryrd}
\usepackage[all]{xy}

\allowdisplaybreaks

\newcommand{\N}{{\ensuremath{\mathbb N}}}

\newcommand{\R}{{\ensuremath{\mathbb R}}}
\newcommand{\K}{{\ensuremath{\mathbb K}}}

\newcommand{\Ccal}{\ensuremath{\mathcal{C}}}

\newcommand{\Lcal}{{\ensuremath{\mathcal{L}}}}

\newcommand{\Ncal}{{\ensuremath{\mathcal{N}}}}
\newcommand{\Ocal}{\ensuremath{\mathcal{O}}}
\newcommand{\Pcal}{\ensuremath{\mathcal{P}}}
\newcommand{\Qcal}{\ensuremath{\mathcal{Q}}}

\newcommand{\Tcal}{\ensuremath{\mathcal{T}}}

\newcommand{\gfrak}{{\ensuremath{\mathfrak{g}}}}

\newcommand{\Lie}{\mathcal{L}ie}
\newcommand{\Com}{\mathcal{C}om}
\newcommand{\Prelie}{\mathcal{P}re\mathcal{L}ie}
\newcommand{\Perm}{\mathcal{P}erm}

\newcommand{\Lieone}{\Lie^1}

\newcommand{\End}{\mathcal{E}nd}

\newcommand{\Nijenhuis}{\ensuremath{\mathcal{N}ij}}
\newcommand{\Binijenhuis}{\ensuremath{\mathcal{B}i\mathcal{N}ij}}
\newcommand{\Preliesquared}{\ensuremath{pre\text{-}\mathcal{L}ie^2}}

\newcommand{\Nijinfty}{Nijenhuis$_\infty$}
\newcommand{\Binijinfty}{bi-Nijenhuis$_\infty$}

\newcommand{\s}{\ensuremath{\mathbb{S}}}
\newcommand{\smodule}{\ensuremath{\s}\text{-module}}

\newcommand{\Free}{\mathcal{F}}
\newcommand{\QO}[2]{\Free(#1)/ ( #2 )}

\newcommand{\FN}[2]{[#1,#2]_{_{\text{F-N}}}}
\newcommand{\FNh}[2]{[#1,#2]_{_{\text{F-N$_\hslash$}}}}

\newcommand{\p}{\partial}
\newcommand{\ddt}[1]{\frac{\p\hspace{8pt}}{\p t^{#1}}}

\newcommand{\sdt}{\gamma}

\newcommand{\ol}[1]{{\ensuremath{\overline{#1}}}}

\newcommand{\Qf}{\hat{Q}}
\newcommand{\Jf}{\hat{J}}
\newcommand{\Jfa}{\hat{J}}
\newcommand{\Jfb}{\hat{K}}

\newcommand{\pvfs}{\wedge^\bullet\Tcal_V}

\newcommand{\vfms}{\Omega^\bullet_V\otimes\Tcal_V}
\newcommand{\vfmsi}[1]{\Omega^{#1}_V\otimes\Tcal_V}
\newcommand{\vfmsh}{(\vfms)\llbracket\hslash\rrbracket}
\newcommand{\vfmshi}[1]{(\vfmsi{#1})\llbracket\hslash\rrbracket}

\newcommand{\cz}{\scriptscriptstyle{\lor}}
\newcommand{\antishriek}{\text{\normalfont{<}}}

\DeclareMathOperator{\coker}{coker}

\DeclareMathOperator{\Ho}{H}
\DeclareMathOperator{\Hom}{Hom}
\DeclareMathOperator{\Ind}{Ind}
\DeclareMathOperator{\sgn}{sgn}

\DeclareMathOperator{\Id}{Id}
\DeclareMathOperator{\ELL}{L}
\DeclareMathOperator{\BC}{B}
\DeclareMathOperator{\CBC}{\Omega}

\DeclareMathOperator{\tot}{tot}

\newtheorem*{lem*}{Lemma}
\newtheorem{prop}[subsubsection]{Proposition}
\newtheorem*{prop*}{Proposition}
\newtheorem{thm}[subsubsection]{Theorem}
\newtheorem*{thm*}{Theorem}
\newtheorem{cor}[subsubsection]{Corollary}

\newtheorem{thmalpha}{Theorem}

\theoremstyle{definition}

\newtheorem*{rem*}{Remark}
\newtheorem*{rems*}{Remarks}

\newtheorem*{defn*}{Definition}

\newtheorem*{ex*}{Example}

\newcommand{\bbone}{1\hspace{-2.6pt}\mbox{\normalfont{l}}}
\newcommand{\half}{{\ensuremath{\frac{1}{2}}}}

\newcommand{\iso}{\cong}

\newcommand{\bhs}{\hspace{10pt}}  
\newcommand{\mhs}{\hspace{7.5pt}} 
\newcommand{\shs}{\hspace{5pt}}   
\newcommand{\abshs}{\hspace{4pt}}   

\newcommand{\leftsub}[2]{{\vphantom{#2}}_{#1}{#2}}

\newcommand{\nijenhuisboxcorolla}[3]{\ensuremath{
 \xygraph{
!{<0pt,0pt>;<4.5pt,0pt>:<0pt,4.5pt>::}
!{(-8,4)}="a"
!{(-6,4)}="ab"
!{(-4,4)}="ba"
!{(-2,4)}="b"
!{(2,4)}="c"
!{(4,4)}="cd"
!{(6,4)}="dc"
!{(8,4)}="d"
!{(-2,1)}="e"
!{(-1.5,1)}="ef"
!{(-1,1)}="fe"
!{(-0.5,1)}="f"
!{(0.5,1)}="g"
!{(1,1)}="gh"
!{(1.5,1)}="hg"
!{(2,1)}="h"
!{(0,0)}*{\scriptscriptstyle #3}
!{(-2,-1)}="i"
!{(0,-1)}="j"
!{(2,-1)}="k"
!{(0,-4)}="l"
!{(-5.1,6)}*{\overbrace{\hspace{27pt}}^{#1}}
!{(5.1,6)}*{\overbrace{\hspace{27pt}}^{#2}}
"a"-"e"
"ab"-"ef"
"ba"-"fe"
"b"-"f"
"c"-@{..}"g"
"cd"-@{..}"gh"
"dc"-@{..}"hg"
"d"-@{..}"h"
"e"-"h"
"e"-"i"
"h"-"k"
"i"-"k"
"j"-"l"
}
}}

\newcommand{\nijenhuissplitboxcorollaA}{\ensuremath{
 \xygraph{
!{<0pt,0pt>;<4.5pt,0pt>:<0pt,4.5pt>::}
!{(-7,7)}="a"
!{(-5,7)}="b"
!{(-3,7)}="c"
!{(-1,7)}="d"
!{(1.5,7)}="e"
!{(3,7)}="f"
!{(5,7)}="g"
!{(7,7)}="h"
!{(9,7)}="i"
!{(-1,4)}="j"
!{(-0.5,4)}="k"
!{(0,4)}="l"
!{(0.5,4)}="m"
!{(1.25,4)}="n"
!{(1.5,4)}="o"
!{(2,4)}="p"
!{(2.5,4)}="q"
!{(3,4)}="r"
!{(-1,2)}="s"
!{(1,2)}="t"
!{(3,2)}="u"
!{(-9,-1)}="A"
!{(-7,-1)}="B"
!{(-5,-1)}="C"
!{(-3,-1)}="D"
!{(0,-1)}="E"
!{(1,-1)}="F"
!{(3,-1)}="G"
!{(5,-1)}="H"
!{(7,-1)}="I"
!{(-3,-4)}="J"
!{(-2.5,-4)}="K"
!{(-2,-4)}="L"
!{(-1.5,-4)}="M"
!{(-0.75,-4)}="N"
!{(-0.5,-4)}="O"
!{(0,-4)}="P"
!{(0.5,-4)}="Q"
!{(1,-4)}="R"
!{(-3,-6)}="S"
!{(-1,-6)}="T"
!{(1,-6)}="U"
!{(-1,-9)}="V"
!{(-4.1,9)}*{\overbrace{\hspace{27pt}}^{I_2}}
!{(1.5,9)}*{\scriptstyle J_2}
!{(6.1,9)}*{\overbrace{\hspace{27pt}}^{J_3}}
!{(-6.1,1)}*{\overbrace{\hspace{27pt}}^{I_1}}
!{(4.1,1)}*{\overbrace{\hspace{27pt}}^{\hspace{4pt} J_1}}
!{(-1,-5)}*{\scriptstyle k_1}
!{(1,3)}*{\scriptstyle k_2}
"a"-"j"
"b"-"k"
"c"-"l"
"d"-"m"
"e"-"n"
"f"-@^{..}"o"
"g"-@^{..}"p"
"h"-@^{..}"q"
"i"-@^{..}"r"
"j"-"r"
"j"-"s"
"r"-"u"
"s"-"u"
"t"-"E"
"A"-"J"
"B"-"K"
"C"-"L"
"D"-"M"
"E"-@^{..}"N"
"F"-@^{..}"O"
"G"-@^{..}"P"
"H"-@^{..}"Q"
"I"-@^{..}"R"
"J"-"R"
"J"-"S"
"R"-"U"
"S"-"U"
"T"-"V"
}
}}

\newcommand{\nijenhuissplitboxcorollaB}{\ensuremath{
 \xygraph{
!{<0pt,0pt>;<-4.5pt,0pt>:<0pt,4.5pt>::}
!{(-7,7)}="a"
!{(-5,7)}="b"
!{(-3,7)}="c"
!{(-1,7)}="d"
!{(1.5,7)}="e"
!{(3,7)}="f"
!{(5,7)}="g"
!{(7,7)}="h"
!{(9,7)}="i"
!{(-1,4)}="j"
!{(-0.5,4)}="k"
!{(0,4)}="l"
!{(0.5,4)}="m"
!{(1.25,4)}="n"
!{(1.5,4)}="o"
!{(2,4)}="p"
!{(2.5,4)}="q"
!{(3,4)}="r"
!{(-1,2)}="s"
!{(1,2)}="t"
!{(3,2)}="u"
!{(-9,-1)}="A"
!{(-7,-1)}="B"
!{(-5,-1)}="C"
!{(-3,-1)}="D"
!{(0,-1)}="E"
!{(1,-1)}="F"
!{(3,-1)}="G"
!{(5,-1)}="H"
!{(7,-1)}="I"
!{(-3,-4)}="J"
!{(-2.5,-4)}="K"
!{(-2,-4)}="L"
!{(-1.5,-4)}="M"
!{(-0.75,-4)}="N"
!{(-0.5,-4)}="O"
!{(0,-4)}="P"
!{(0.5,-4)}="Q"
!{(1,-4)}="R"
!{(-3,-6)}="S"
!{(-1,-6)}="T"
!{(1,-6)}="U"
!{(-1,-9)}="V"
!{(-4.1,9)}*{\overbrace{\hspace{27pt}}^{J_2}}
!{(6.1,9)}*{\overbrace{\hspace{27pt}}^{I_2}}
!{(-6.1,1)}*{\overbrace{\hspace{27pt}}^{J_1}}
!{(4.1,1)}*{\overbrace{\hspace{27pt}}^{ I_1} \hspace{4pt}}
!{(-1,-5)}*{\scriptstyle k_1}
!{(1,3)}*{\scriptstyle k_2}
"a"-@^{..}"j"
"b"-@^{..}"k"
"c"-@^{..}"l"
"d"-@^{..}"m"
"f"-"o"
"g"-"p"
"h"-"q"
"i"-"r"
"j"-"r"
"j"-"s"
"r"-"u"
"s"-"u"
"t"-"E"
"A"-@^{..}"J"
"B"-@^{..}"K"
"C"-@^{..}"L"
"D"-@^{..}"M"
"E"-"N"
"F"-"O"
"G"-"P"
"H"-"Q"
"I"-"R"
"J"-"R"
"J"-"S"
"R"-"U"
"S"-"U"
"T"-"V"
}
}}

\newcommand{\PL}[2]{\ensuremath{
 \xygraph{
!{<0pt,0pt>;<4pt,0pt>:<0pt,4pt>::}
!{(-2,2.5)}*{\scriptscriptstyle #1}
!{(2,2.5)}*{\scriptscriptstyle #2}
!{(-2,1.5)}="a"
!{(2,1.5)}="b"
!{(0,-0.5)}*{\scriptscriptstyle{\circ}}="c"
!{(0,-2.5)}="d"
"a"-"c"
"b"-@^{..}"c"
"c"-"d"
}
}}

\newcommand{\PLop}[2]{\ensuremath{
 \xygraph{
!{<0pt,0pt>;<4pt,0pt>:<0pt,4pt>::}
!{(-2,2.5)}*{\scriptscriptstyle #1}
!{(2,2.5)}*{\scriptscriptstyle #2}
!{(-2,1.5)}="a"
!{(2,1.5)}="b"
!{(0,-0.5)}*{\scriptscriptstyle{\circ}}="c"
!{(0,-2.5)}="d"
"a"-@^{..}"c"
"b"-"c"
"c"-"d"
}
}}

\newcommand{\Ysmall}{\ensuremath{
 \xygraph{
!{<0pt,0pt>;<2.5pt,0pt>:<0pt,-2.5pt>::}
!{(0,1.5)}="a"
!{(0,-0.5)}="b"
!{(-2,-2.5)}="c"
!{(2,-2.5)}="d"
"a"-"b"
"b"-"c"
"b"-"d"
}
}}

\newcommand{\PLsmall}{\ensuremath{
 \xygraph{
!{<0pt,0pt>;<2.5pt,0pt>:<0pt,-2.5pt>::}
!{(0,1.5)}="a"
!{(0,-0.5)}="b"
!{(-2,-2.5)}="c"
!{(2,-2.5)}="d"
"a"-"b"
"b"-"c"
"b"-@^{..}"d"
}
}}

\newcommand{\PLopsmall}{\ensuremath{
 \xygraph{
!{<0pt,0pt>;<2.5pt,0pt>:<0pt,-2.5pt>::}
!{(0,1.5)}="a"
!{(0,-0.5)}="b"
!{(-2,-2.5)}="c"
!{(2,-2.5)}="d"
"a"-"b"
"b"-@^{..}"c"
"b"-"d"
}
}}

\newcommand{\PLWsmall}{\ensuremath{
 \xygraph{
!{<0pt,0pt>;<2.5pt,0pt>:<0pt,-2.5pt>::}
!{(0,1.5)}="a"
!{(0,-0.5)}*{\scriptscriptstyle{\circ}}="b"
!{(-2,-2.5)}="c"
!{(2,-2.5)}="d"
"a"-"b"
"b"-"c"
"b"-@^{..}"d"
}
}}

\newcommand{\PLWopsmall}{\ensuremath{
 \xygraph{
!{<0pt,0pt>;<2.5pt,0pt>:<0pt,-2.5pt>::}
!{(0,1.5)}="a"
!{(0,-0.5)}*{\scriptscriptstyle{\circ}}="b"
!{(-2,-2.5)}="c"
!{(2,-2.5)}="d"
"a"-"b"
"b"-@^{..}"c"
"b"-"d"
}
}}

\newcommand{\PLBsmall}{\ensuremath{
 \xygraph{
!{<0pt,0pt>;<2.5pt,0pt>:<0pt,-2.5pt>::}
!{(0,1.5)}="a"
!{(0,-0.5)}*{\scriptscriptstyle{\bullet}}="b"
!{(-2,-2.5)}="c"
!{(2,-2.5)}="d"
"a"-"b"
"b"-"c"
"b"-@^{..}"d"
}
}}

\newcommand{\PLBopsmall}{\ensuremath{
 \xygraph{
!{<0pt,0pt>;<2.5pt,0pt>:<0pt,-2.5pt>::}
!{(0,1.5)}="a"
!{(0,-0.5)}*{\scriptscriptstyle{\bullet}}="b"
!{(-2,-2.5)}="c"
!{(2,-2.5)}="d"
"a"-"b"
"b"-@^{..}"c"
"b"-"d"
}
}}

\newcommand{\XX}[5]{\ensuremath{
 \xygraph{
!{<0pt,0pt>;<4pt,0pt>:<0pt,4pt>::}
!{(-3,3.5)}*{\scriptscriptstyle #1}
!{(1,3.5)}*{\scriptscriptstyle #2}
!{(-3,2.5)}="a"
!{(1,2.5)}="b"
!{(3,1.5)}*{\scriptscriptstyle #3}
!{(-1,0.5)}*{\scriptscriptstyle #4}="c"
!{(3,0.5)}="d"
!{(1,-1.5)}*{\scriptscriptstyle #5}="e"
!{(1,-3.5)}="f"
"a"-"c"
"b"-"c"
"c"-"e"
"d"-"e"
"e"-"f"
}
}}

\newcommand{\YY}[3]{\XX{#1}{#2}{#3}{{}}{{}}}

\newcommand{\XXop}[5]{\ensuremath{
 \xygraph{
!{<0pt,0pt>;<-4pt,0pt>:<0pt,4pt>::}
!{(-3,3.5)}*{\scriptscriptstyle #3}
!{(1,3.5)}*{\scriptscriptstyle #2}
!{(-3,2.5)}="a"
!{(1,2.5)}="b"
!{(3,1.5)}*{\scriptscriptstyle #1}
!{(-1,0.5)}*{\scriptscriptstyle #4}="c"
!{(3,0.5)}="d"
!{(1,-1.5)}*{\scriptscriptstyle #5}="e"
!{(1,-3.5)}="f"
"a"-"c"
"b"-"c"
"c"-"e"
"d"-"e"
"e"-"f"
}
}}

\newcommand{\YYop}[3]{\XXop{#1}{#2}{#3}{{}}{{}}}

\newcommand{\PLXPLX}[5]{\ensuremath{
 \xygraph{
!{<0pt,0pt>;<4pt,0pt>:<0pt,4pt>::}
!{(-3,3.5)}*{\scriptscriptstyle #1}
!{(1,3.5)}*{\scriptscriptstyle #2}
!{(-3,2.5)}="a"
!{(1,2.5)}="b"
!{(3,1.5)}*{\scriptscriptstyle #3}
!{(-1,0.5)}*{\scriptscriptstyle #4}="c"
!{(3,0.5)}="d"
!{(1,-1.5)}*{\scriptscriptstyle #5}="e"
!{(1,-3.5)}="f"
"a"-"c"
"b"-@^{..}"c"
"c"-"e"
"d"-@^{..}"e"
"e"-"f"
}
}}

\newcommand{\PLYPLY}[3]{\PLXPLX{#1}{#2}{#3}{{}}{{}}}

\newcommand{\PLWPLW}[3]{\PLXPLX{#1}{#2}{#3}{\circ}{\circ}}
\newcommand{\PLWPLB}[3]{\PLXPLX{#1}{#2}{#3}{\circ}{\bullet}}

\newcommand{\PLBPLW}[3]{\PLXPLX{#1}{#2}{#3}{\bullet}{\circ}}
\newcommand{\PLBPLB}[3]{\PLXPLX{#1}{#2}{#3}{\bullet}{\bullet}}

\newcommand{\PLXPLXop}[5]{\ensuremath{
 \xygraph{
!{<0pt,0pt>;<-4pt,0pt>:<0pt,4pt>::}
!{(-3,3.5)}*{\scriptscriptstyle #3}
!{(1,3.5)}*{\scriptscriptstyle #2}
!{(-3,2.5)}="a"
!{(1,2.5)}="b"
!{(3,1.5)}*{\scriptscriptstyle #1}
!{(-1,0.5)}*{\scriptscriptstyle #4}="c"
!{(3,0.5)}="d"
!{(1,-1.5)}*{\scriptscriptstyle #5}="e"
!{(1,-3.5)}="f"
"a"-@^{..}"c"
"b"-"c"
"c"-@^{..}"e"
"d"-"e"
"e"-"f"
}
}}

\newcommand{\PLYPLYop}[3]{\PLXPLXop{#1}{#2}{#3}{{}}{{}}}

\newcommand{\PLWPLWop}[3]{\PLXPLXop{#1}{#2}{#3}{\circ}{\circ}}
\newcommand{\PLWPLBop}[3]{\PLXPLXop{#1}{#2}{#3}{\circ}{\bullet}}

\newcommand{\PLBPLWop}[3]{\PLXPLXop{#1}{#2}{#3}{\bullet}{\circ}}
\newcommand{\PLBPLBop}[3]{\PLXPLXop{#1}{#2}{#3}{\bullet}{\bullet}}

\newcommand{\PLXX}[5]{\ensuremath{
 \xygraph{
!{<0pt,0pt>;<4pt,0pt>:<0pt,4pt>::}
!{(-3,3.5)}*{\scriptscriptstyle #1}
!{(1,3.5)}*{\scriptscriptstyle #2}
!{(-3,2.5)}="a"
!{(1,2.5)}="b"
!{(3,1.5)}*{\scriptscriptstyle #3}
!{(-1,0.5)}*{\scriptscriptstyle #4}="c"
!{(3,0.5)}="d"
!{(1,-1.5)}*{\scriptscriptstyle #5}="e"
!{(1,-3.5)}="f"
"a"-"c"
"b"-@^{..}"c"
"c"-"e"
"d"-"e"
"e"-"f"
}
}}

\newcommand{\PLWY}[3]{\PLXX{#1}{#2}{#3}{\circ}{{}}}

\newcommand{\PLBY}[3]{\PLXX{#1}{#2}{#3}{\bullet}{{}}}
\newcommand{\PLBW}[3]{\PLXX{#1}{#2}{#3}{\bullet}{\circ}}

\newcommand{\PLXXop}[5]{\ensuremath{
 \xygraph{
!{<0pt,0pt>;<-4pt,0pt>:<0pt,4pt>::}
!{(-3,3.5)}*{\scriptscriptstyle #3}
!{(1,3.5)}*{\scriptscriptstyle #2}
!{(-3,2.5)}="a"
!{(1,2.5)}="b"
!{(3,1.5)}*{\scriptscriptstyle #1}
!{(-1,0.5)}*{\scriptscriptstyle #4}="c"
!{(3,0.5)}="d"
!{(1,-1.5)}*{\scriptscriptstyle #5}="e"
!{(1,-3.5)}="f"
"a"-@^{..}"c"
"b"-"c"
"c"-"e"
"d"-"e"
"e"-"f"
}
}}

\newcommand{\PLYYop}[3]{\PLXXop{#1}{#2}{#3}{{}}{{}}}

\newcommand{\PLWYop}[3]{\PLXXop{#1}{#2}{#3}{\circ}{{}}}

\newcommand{\PLBYop}[3]{\PLXXop{#1}{#2}{#3}{\bullet}{{}}}

\newcommand{\XPLX}[5]{\ensuremath{
 \xygraph{
!{<0pt,0pt>;<4pt,0pt>:<0pt,4pt>::}
!{(-3,3.5)}*{\scriptscriptstyle #1}
!{(1,3.5)}*{\scriptscriptstyle #2}
!{(-3,2.5)}="a"
!{(1,2.5)}="b"
!{(3,1.5)}*{\scriptscriptstyle #3}
!{(-1,0.5)}*{\scriptscriptstyle #4}="c"
!{(3,0.5)}="d"
!{(1,-1.5)}*{\scriptscriptstyle #5}="e"
!{(1,-3.5)}="f"
"a"-"c"
"b"-"c"
"c"-"e"
"d"-@^{..}"e"
"e"-"f"
}
}}

\newcommand{\YPLY}[3]{\XPLX{#1}{#2}{#3}{{}}{{}}}
\newcommand{\YPLW}[3]{\XPLX{#1}{#2}{#3}{{}}{\circ}}
\newcommand{\YPLB}[3]{\XPLX{#1}{#2}{#3}{{}}{\bullet}}

\newcommand{\XPLXop}[5]{\ensuremath{
 \xygraph{
!{<0pt,0pt>;<-4pt,0pt>:<0pt,4pt>::}
!{(-3,3.5)}*{\scriptscriptstyle #3}
!{(1,3.5)}*{\scriptscriptstyle #2}
!{(-3,2.5)}="a"
!{(1,2.5)}="b"
!{(3,1.5)}*{\scriptscriptstyle #1}
!{(-1,0.5)}*{\scriptscriptstyle #4}="c"
!{(3,0.5)}="d"
!{(1,-1.5)}*{\scriptscriptstyle #5}="e"
!{(1,-3.5)}="f"
"a"-"c"
"b"-"c"
"c"-@^{..}"e"
"d"-"e"
"e"-"f"
}
}}

\newcommand{\YPLYop}[3]{\XPLXop{#1}{#2}{#3}{{}}{{}}}
\newcommand{\YPLWop}[3]{\XPLXop{#1}{#2}{#3}{{}}{\circ}}
\newcommand{\YPLBop}[3]{\XPLXop{#1}{#2}{#3}{{}}{\bullet}}

\newcommand{\PLXtwX}[5]{\ensuremath{
 \xygraph{
!{<0pt,0pt>;<4pt,0pt>:<0pt,4pt>::}
!{(-3,3.5)}*{\scriptscriptstyle #1}
!{(1,3.5)}*{\scriptscriptstyle #2}
!{(-3,2.5)}="a"
!{(1,2.5)}="b"
!{(3,1.5)}*{\scriptscriptstyle #3}
!{(-1,0.5)}*{\scriptscriptstyle #4}="c"
!{(3,0.5)}="d"
!{(1,-1.5)}*{\scriptscriptstyle #5}="e"
!{(1,-3.5)}="f"
"a"-@^{..}"c"
"b"-"c"
"c"-"e"
"d"-"e"
"e"-"f"
}
}}

\newcommand{\PLYtwY}[3]{\PLXtwX{#1}{#2}{#3}{{}}{{}}}

\newcommand{\PLWtwY}[3]{\PLXtwX{#1}{#2}{#3}{\circ}{{}}}

\newcommand{\PLBtwY}[3]{\PLXtwX{#1}{#2}{#3}{\bullet}{{}}}

\newcommand{\PLXtwPLX}[5]{\ensuremath{
 \xygraph{
!{<0pt,0pt>;<4pt,0pt>:<0pt,4pt>::}
!{(-3,3.5)}*{\scriptscriptstyle #1}
!{(1,3.5)}*{\scriptscriptstyle #2}
!{(-3,2.5)}="a"
!{(1,2.5)}="b"
!{(3,1.5)}*{\scriptscriptstyle #3}
!{(-1,0.5)}*{\scriptscriptstyle #4}="c"
!{(3,0.5)}="d"
!{(1,-1.5)}*{\scriptscriptstyle #5}="e"
!{(1,-3.5)}="f"
"a"-@^{..}"c"
"b"-"c"
"c"-"e"
"d"-@^{..}"e"
"e"-"f"
}
}}

\newcommand{\PLYtwPLY}[3]{\PLXtwPLX{#1}{#2}{#3}{{}}{{}}}

\newcommand{\PLWtwPLW}[3]{\PLXtwPLX{#1}{#2}{#3}{\circ}{\circ}}

\newcommand{\PLBtwPLB}[3]{\PLXtwPLX{#1}{#2}{#3}{\bullet}{\bullet}}

\newcommand{\YWBcorolla}{\ensuremath{
 \xygraph{
!{<0pt,0pt>;<4pt,0pt>:<0pt,4pt>::}
!{(-7,7)}="a"
!{(-3,7)}="b"
!{(-5,5)}="c"
!{(-1,5)}="d"
!{(-4.5,4.5)}*{.}="cea"
!{(-4,4)*{.}}="ceb"
!{(-3.5,3.5)}*{.}="cec"
!{(-3,3)}="e"
!{(1,3)}="f"
!{(-1,1)}*{\scriptscriptstyle{\circ}}="g"
!{(-0.5,0.5)}*{.}="gha"
!{(0,0)}*{.}="ghb"
!{(0.5,-0.5)}*{.}="ghc"
!{(3,1)}="h"
!{(1,-1)}*{\scriptscriptstyle{\circ}}="i"
!{(5,-1)}="j"
!{(3,-3)}*{\scriptscriptstyle{\bullet}}="k"
!{(3.5,-3.5)}*{.}="kma"
!{(4,-4)}*{.}="kmb"
!{(4.5,-4.5)}*{.}="kmc"
!{(7,-3)}="l"
!{(5,-5)}*{\scriptscriptstyle{\bullet}}="m"
!{(5,-7)}="n"
!{(-7,8)}*{\scriptscriptstyle{\sigma(1)}}="A"
!{(-3,8)}*{\scriptscriptstyle{\sigma(2)}}="B"
!{(-1,6)}*{\scriptscriptstyle{\sigma(i)}}="D"
!{(2,4)}*{\scriptscriptstyle{\sigma(i+1})}="F"
!{(4,2)}*{\scriptscriptstyle{\sigma(i+j)}}="H"
!{(7,0)}*{\scriptscriptstyle{\sigma(i+j+1)}}="J"
!{(7,-2)}*{\scriptscriptstyle{\sigma(n)}}="L"
"a"-"c"
"b"-"c"
"d"-"e"
"e"-"g"
"f"-@^{..}"g"
"h"-@^{..}"i"
"i"-"k"
"j"-@^{..}"k"
"l"-@^{..}"m"
"m"-"n"
}
}}

\newcommand{\YWBcorollab}{\ensuremath{
 \xygraph{
!{<0pt,0pt>;<4pt,0pt>:<0pt,4pt>::}
!{(-7,7)}="a"
!{(-3,7)}="b"
!{(-5,5)}="c"
!{(-1,5)}="d"
!{(-4.5,4.5)}*{.}="cea"
!{(-4,4)*{.}}="ceb"
!{(-3.5,3.5)}*{.}="cec"
!{(-3,3)}="e"
!{(1,3)}="f"
!{(-1,1)}*{\scriptscriptstyle{\circ}}="g"
!{(-0.5,0.5)}*{.}="gha"
!{(0,0)}*{.}="ghb"
!{(0.5,-0.5)}*{.}="ghc"
!{(3,1)}="h"
!{(1,-1)}*{\scriptscriptstyle{\circ}}="i"
!{(5,-1)}="j"
!{(3,-3)}*{\scriptscriptstyle{\bullet}}="k"
!{(3.5,-3.5)}*{.}="kma"
!{(4,-4)}*{.}="kmb"
!{(4.5,-4.5)}*{.}="kmc"
!{(7,-3)}="l"
!{(5,-5)}*{\scriptscriptstyle{\bullet}}="m"
!{(5,-7)}="n"
!{(-7,8)}*{\scriptscriptstyle{\sigma(1)}}="A"
!{(-3,8)}*{\scriptscriptstyle{\sigma(2)}}="B"
!{(-1,6)}*{\scriptscriptstyle{\sigma(i)}}="D"
!{(2,4)}*{\scriptscriptstyle{\sigma(i+1})}="F"
!{(4,2)}*{\scriptscriptstyle{\sigma(i+k)}}="H"
!{(7,0)}*{\scriptscriptstyle{\sigma(i+k+1)}}="J"
!{(7,-2)}*{\scriptscriptstyle{\sigma(n)}}="L"
"a"-"c"
"b"-"c"
"d"-"e"
"e"-"g"
"f"-@^{..}"g"
"h"-@^{..}"i"
"i"-"k"
"j"-@^{..}"k"
"l"-@^{..}"m"
"m"-"n"
}
}}

\newcommand{\nijPBWa}{\ensuremath{
 \xygraph{
!{<0pt,0pt>;<4pt,0pt>:<0pt,4pt>::}
!{(-6,6)}="a"
!{(-2,6)}="b"
!{(-4,4)}="c"
!{(-3.5,3.5)}*{.}="cda"
!{(-3,3)*{.}}="cdb"
!{(-2.5,2.5)}*{.}="cdc"
!{(-2,2)}="da"
!{(2,2)}="db"
!{(0,0)}="e"
!{(4,0)}="f"
!{(2,-2)}="g"
!{(2.5,-2.5)}*{.}="gia"
!{(3,-3)}*{.}="gib"
!{(3.5,-3.5)}*{.}="gic"
!{(4,-4)}="ha"
!{(8,-4)}="hb"
!{(6,-6)}="i"
!{(6,-8)}="j"
!{(-6,7)}*{\scriptscriptstyle{1}}="A"
!{(-2,7)}*{\scriptscriptstyle{i_1}}="B"
!{(2,3)}*{\scriptscriptstyle�{i_r}}="DB"
!{(4,1)}*{\scriptscriptstyle{j_1}}="F"
!{(8,-3)}*{\scriptscriptstyle{j_s}}="H"
"a"-"c"
"b"-"c"
"da"-"e"
"db"-"e"
"e"-"g"
"f"-@^{..}"g"
"ha"-"i"
"hb"-@^{..}"i"
"i"-"j"
}
}}

\newcommand{\nijPBWb}{\ensuremath{
 \xygraph{
!{<0pt,0pt>;<4pt,0pt>:<0pt,4pt>::}
!{(-8,8)}="y"
!{(-4,8)}="z"
!{(-6,6)}="a"
!{(-2,6)}="b"
!{(-4,4)}="c"
!{(-3.5,3.5)}*{.}="cda"
!{(-3,3)*{.}}="cdb"
!{(-2.5,2.5)}*{.}="cdc"
!{(-2,2)}="da"
!{(2,2)}="db"
!{(0,0)}="e"
!{(4,0)}="f"
!{(2,-2)}="g"
!{(2.5,-2.5)}*{.}="gia"
!{(3,-3)}*{.}="gib"
!{(3.5,-3.5)}*{.}="gic"
!{(4,-4)}="ha"
!{(8,-4)}="hb"
!{(6,-6)}="i"
!{(6,-8)}="j"
!{(-8,9)}*{\scriptscriptstyle{1}}="Y"
!{(-4,9)}*{\scriptscriptstyle{i_1}}="Z"
!{(-2,7)}*{\scriptscriptstyle{i_2}}="B"
!{(2,3)}*{\scriptscriptstyle�{i_r}}="D"
!{(4,1)}*{\scriptscriptstyle{j_1}}="F"
!{(8,-3)}*{\scriptscriptstyle{j_s}}="H"
"y"-@^{..}"a"
"z"-"a"
"a"-"c"
"b"-"c"
"da"-"e"
"db"-"e"
"e"-"g"
"f"-@^{..}"g"
"ha"-"i"
"hb"-@^{..}"i"
"i"-"j"
}
}}

\newcommand{\binijPBWa}{\ensuremath{
 \xygraph{
!{<0pt,0pt>;<4pt,0pt>:<0pt,4pt>::}
!{(-7,7)}="a"
!{(-3,7)}="b"
!{(-5,5)}="c"
!{(-1,5)}="d"
!{(-4.5,4.5)}*{.}="cea"
!{(-4,4)*{.}}="ceb"
!{(-3.5,3.5)}*{.}="cec"
!{(-3,3)}="e"
!{(1,3)}="f"
!{(-1,1)}*{\scriptscriptstyle{\circ}}="g"
!{(-0.5,0.5)}*{.}="gha"
!{(0,0)}*{.}="ghb"
!{(0.5,-0.5)}*{.}="ghc"
!{(3,1)}="h"
!{(1,-1)}*{\scriptscriptstyle{\circ}}="i"
!{(5,-1)}="j"
!{(3,-3)}*{\scriptscriptstyle{\bullet}}="k"
!{(3.5,-3.5)}*{.}="kma"
!{(4,-4)}*{.}="kmb"
!{(4.5,-4.5)}*{.}="kmc"
!{(7,-3)}="l"
!{(5,-5)}*{\scriptscriptstyle{\bullet}}="m"
!{(5,-7)}="n"
!{(-7,8)}*{\scriptscriptstyle{1}}="A"
!{(-3,8)}*{\scriptscriptstyle{i_1}}="B"
!{(-1,6)}*{\scriptscriptstyle{i_r}}="D"
!{(2,4)}*{\scriptscriptstyle{j_1}}="F"
!{(4,2)}*{\scriptscriptstyle{j_s}}="H"
!{(7,0)}*{\scriptscriptstyle{j_{s+1}}}="J"
!{(8,-2)}*{\scriptscriptstyle{j_{s+t}}}="L"
"a"-"c"
"b"-"c"
"d"-"e"
"e"-"g"
"f"-@^{..}"g"
"h"-@^{..}"i"
"i"-"k"
"j"-@^{..}"k"
"l"-@^{..}"m"
"m"-"n"
}
}}

\newcommand{\binijPBWb}{\ensuremath{
 \xygraph{
!{<0pt,0pt>;<4pt,0pt>:<0pt,4pt>::}
!{(-9,9)}="x"
!{(-5,9)}="y"
!{(-7,7)}*{\scriptscriptstyle{\circ}}="a"
!{(-3,7)}="b"
!{(-5,5)}="c"
!{(-1,5)}="d"
!{(-4.5,4.5)}*{.}="cea"
!{(-4,4)*{.}}="ceb"
!{(-3.5,3.5)}*{.}="cec"
!{(-3,3)}="e"
!{(1,3)}="f"
!{(-1,1)}*{\scriptscriptstyle{\circ}}="g"
!{(-0.5,0.5)}*{.}="gha"
!{(0,0)}*{.}="ghb"
!{(0.5,-0.5)}*{.}="ghc"
!{(3,1)}="h"
!{(1,-1)}*{\scriptscriptstyle{\circ}}="i"
!{(5,-1)}="j"
!{(3,-3)}*{\scriptscriptstyle{\bullet}}="k"
!{(3.5,-3.5)}*{.}="kma"
!{(4,-4)}*{.}="kmb"
!{(4.5,-4.5)}*{.}="kmc"
!{(7,-3)}="l"
!{(5,-5)}*{\scriptscriptstyle{\bullet}}="m"
!{(5,-7)}="n"
!{(-9,10)}*{\scriptscriptstyle{1}}="X"
!{(-5,10)}*{\scriptscriptstyle{i_1}}="Y"
!{(-3,8)}*{\scriptscriptstyle{i_2}}="B"
!{(-1,6)}*{\scriptscriptstyle{i_r}}="D"
!{(2,4)}*{\scriptscriptstyle{j_1}}="F"
!{(4,2)}*{\scriptscriptstyle{j_s}}="H"
!{(7,0)}*{\scriptscriptstyle{j_{s+1}}}="J"
!{(8,-2)}*{\scriptscriptstyle{j_{s+t}}}="L"
"x"-@^{..}"a"
"y"-"a"
"a"-"c"
"b"-"c"
"d"-"e"
"e"-"g"
"f"-@^{..}"g"
"h"-@^{..}"i"
"i"-"k"
"j"-@^{..}"k"
"l"-@^{..}"m"
"m"-"n"
}
}}

\newcommand{\binijPBWc}{\ensuremath{
 \xygraph{
!{<0pt,0pt>;<4pt,0pt>:<0pt,4pt>::}
!{(-6,6)}="y"
!{(-2,6)}="z"
!{(-4,4)}*{\scriptscriptstyle{\bullet}}="a"
!{(0,4)}="b"
!{(-2,2)}="c"
!{(2,2)}="d"
!{(-1.5,1.5)}*{.}="cea"
!{(-1,1)*{.}}="ceb"
!{(-0.5,0.5)}*{.}="cec"
!{(0,0)}="e"
!{(4,0)}="f"
!{(2,-2)}*{\scriptscriptstyle{\bullet}}="g"
!{(6,-2)}="h"
!{(2.5,-2.5)}*{.}="gia"
!{(3,-3)}*{.}="gib"
!{(3.5,-3.5)}*{.}="gic"
!{(4,-4)}*{\scriptscriptstyle{\bullet}}="i"
!{(4,-6)}="j"
!{(-6,7)}*{\scriptscriptstyle{1}}="Y"
!{(-2,7)}*{\scriptscriptstyle{i_1}}="Z"
!{(0,5)}*{\scriptscriptstyle{i_2}}="B"
!{(2,3)}*{\scriptscriptstyle�{i_r}}="D"
!{(4,1)}*{\scriptscriptstyle{j_1}}="F"
!{(6,-1)}*{\scriptscriptstyle{j_s}}="H"
"y"-@^{..}"a"
"z"-"a"
"a"-"c"
"b"-"c"
"d"-"e"
"e"-"g"
"f"-@^{..}"g"
"h"-@^{..}"i"
"i"-"j"
}
}}

\begin{document}

\begin{abstract}
Recently S.~Merkulov \cite{Merkulov2004a,Merkulov2005,Merkulov2006} established a new link between differential geometry and homological algebra by giving descriptions of several differential geometric structures in terms of algebraic operads and props. In particular he described Nijenhuis structures as corresponding to representations of the cobar construction on the Koszul dual of a certain quadratic operad. In this paper we prove, using the PBW-basis method of E.~Hoffbeck \cite{Hoffbeck2008}, that the operad governing Nijenhuis structures is Koszul, thereby showing that Nijenhuis structures correspond to representations of the minimal resolution of this operad. We also construct an operad such that representations of its minimal resolution in a vector space $V$ are in one-to-one correspondence with pairs of compatible Nijenhuis structures on the formal manifold associated to $V$.
\end{abstract}

\maketitle


\section*{Introduction}

The language of operads and props has since its renaissance in the 90's come to find its way into many fields of mathematics. A recent discovery is that several differential geometric structures can be translated to this language. In \cite{Kontsevich2003} M.~Kontsevich showed that homological vector fields on a formal manifold $V$ correspond to representations in $V$ of the minimal resolution of the operad of Lie algebras. In the papers \cite{Merkulov2004a}, \cite{Merkulov2005}, and \cite{Merkulov2006} S.~Merkulov gave operadic and propic descriptions, called operad and prop profiles, of Hertling-Manin, Nijenhuis, and Poisson structures, respectively. The common characteristic of all these structures is that they can be defined as Maurer-Cartan elements of certain Lie algebras. On the operadic, or propic, side the Maurer-Cartan equations are encoded by quadratic differentials of minimal resolutions of certain operads, or props. This translation has been useful e.g.~in that Merkulov in \cite{Merkulov2008,Merkulov2008a}  gave a propic formulation of the deformation quantization of Poisson structures. In this context wheels (directed cycles) had to be introduced which translates to a condition on traces on the geometric side.

The dictionary between differential geometry and homological algebra was further enlarged in \cite{Strohmayer2008} with the prop profile of bi-Hamiltonian structures, i.e.~pairs of compatible Poisson structures. Representations of this prop in a vector space $V$ correspond to Maurer-Cartan elements in a Lie subalgebra of $\wedge^\bullet\Tcal_V\llbracket\hslash\rrbracket$, where $\hslash$ is a formal parameter.

The vector field valued differential forms $\vfms$ of a manifold $V$ together with the Fr\"olicher-Nijenhuis bracket $\FN{\_}{\_}$ comprise a graded Lie algebra \cite{Nijenhuis1955}. A Nijenhuis structure is a Maurer-Cartan element of this Lie algebra, i.e.~an element $J\in\vfmsi{1}$ such that $\FN{J}{J}=0$. In \cite{Merkulov2005} S.~Merkulov defined a quadratic operad $\Nijenhuis$ such that representations of $\Omega(\Nijenhuis^{\antishriek})$, the cobar construction on the Koszul dual cooperad, in a vector space $V$ correspond to Nijenhuis structures on the formal manifold associated to $V$. If an operad $\Pcal$ is Koszul, the cobar construction on its Koszul dual is a minimal resolution of the operad, denoted by $\Pcal_\infty$. The operad $\Nijenhuis$ consists of the Lie operad and the pre-Lie operad with the Lie bracket and pre-Lie product differing by one in degree and compatible in a certain sense. Until recently the available methods have not been sufficient to prove the Koszulness of $\Nijenhuis$; e.g.~the compatibility relation of the operations does not define a distributive law and the operad does not come from a set theoretic operad, thus neither the methods of \cite{Markl1996} nor \cite{Vallette2007} are applicable. Using the method of PBW-bases for operads, introduced by E.~Hoffbeck in \cite{Hoffbeck2008}, we prove that $\Nijenhuis$ is Koszul. Thereby we obtain the following result:

\begin{thmalpha}
\label{nongradednijenhuisthm}
 There is a one-to-one correspondence between representations of $\Nijenhuis_\infty$ in $\R^n$ and formal Nijenhuis structures on $\R^n$ vanishing at the origin.
\end{thmalpha}

We say that two Nijenhuis structures $J$ and $K$ are compatible if their sum is again a Nijenhuis structure. We call such a pair a bi-Nijenhuis structure. In this paper we show that bi-Nijenhuis structures can be derived from a rather simple algebraic structure. This structure consists of a Lie bracket and two pre-Lie products differing by one in degree with respect to the Lie bracket. The pre-Lie products are compatible in the sense that their sum again is a pre-Lie product and each of them is compatible with the Lie bracket in the sense that they form a $\Nijenhuis$ algebra. We denote the operad encoding such structures $\Binijenhuis$. Again using the PBW-basis method of Hoffbeck we show that $\Binijenhuis$ is Koszul, making it possible to prove the following:

\begin{thmalpha}
\label{introthmnongraded}
 There is a one-to-one correspondence between representations of $\Binijenhuis_\infty$ in $\R^n$ and formal bi-Nijenhuis structures on  $\R^n$ vanishing at the origin.
\end{thmalpha}

Considering representations in arbitrary graded manifolds we obtain the following generalization:

\begin{thmalpha}
\label{introthmgraded}
  There is a one-to-one correspondence between representations of $\Binijenhuis_\infty$ in a graded vector space $V$ and formal power series 
\[\Gamma=\sum_k \Gamma_k\hslash^k\in\vfmsh\] 
satisfying the conditions
\begin{enumerate}
 \item $\Gamma_k\in\vfmsi{\geq k}$,
 \item $|\Gamma|=1$,
 \item $\FN{\Gamma}{\Gamma}=0$,
 \item $\Gamma|_0=0$.
\end{enumerate}
\end{thmalpha}

The paper is organized as follows. In Section \ref{Nijenhuisgeom} we review the geometry of Nijenhuis structures. In Section \ref{extraction} we recall Merkulov's operad profile of Nijenhuis structures and define the operad encoding bi-Nijenhuis structures. In Section \ref{resolution}, we review Hoffbeck's notion of PBW-bases for operads and use this to show that the operads $\Nijenhuis$ and $\Binijenhuis$, of Nijenhuis and compatible Nijenhuis structures, respectively, are Koszul. Thus we prove Theorem \ref{nongradednijenhuisthm}. We also explicitly describe the minimal resolution $\Binijenhuis_\infty$ of $\Binijenhuis$. In Section \ref{geominterpret}, we give a geometrical description of the operad $\Binijenhuis_\infty$ and prove Theorems \ref{introthmnongraded} and \ref{introthmgraded}.

\subsection*{Preliminaries} By $\K$ we denote an arbitrary field of characteristic zero. For a vector space $V$ over $\K$ we denote the linear dual $\Hom(V,\K)$ by $V^*$. The symmetric product of vector spaces is denoted by $\odot$. Let $\s_n$ denote the symmetric group of permutations of the set $\{1,\dotsc,n\}$ and let $\bbone_n$ and $\sgn_n$ denote the trivial and sign representations of $\s_n$, respectively. Throughout the paper we use the Einstein summation convention, i.e.~we always sum over repeated upper and lower indices, $X^a\p_a=\sum_a X^a\p_a$. 


\section{Nijenhuis geometry}
\label{Nijenhuisgeom}
Here we review basic definitions concerning Nijenhuis structures and define a notion of compatibility.

\subsection{Nijenhuis structures}

Let $V$ be a manifold and let $\Tcal_V$ denote the tangent sheaf. To a morphism $J\colon\Tcal_V\to\Tcal_V$ one can associate a morphism $\Ncal_J\colon\wedge^2~\Tcal_V\to\Tcal_V$, called the \emph{Nijenhuis torsion}, defined by
\begin{equation}
\label{nc}
 \Ncal_J(X,Y):=JJ[X,Y]+[JX,JY]-J[X,JY]-J[JX,Y].
\end{equation}
We call an endomorphism $J$ of $\Tcal_V$ a \emph{Nijenhuis structure} if it satisfies $\Ncal_J=0$.

An \emph{almost complex structure} on an even dimensional manifold $V$ is an endomorphism $J$ of $\Tcal_V$ satisfying $J^2=-\Id$. By the Newlander-Nirenberg Theorem \cite{Newlander1957} the vanishing of the Nijenhuis torsion of an almost complex structure $J$ is equivalent to $J$ being a complex structure on $V$. This is probably the most important application of the Nijenhuis torsion. Other examples can be found in \cite{Nijenhuis1955}.

\subsection{Formal graded manifolds}
\label{fgm}

Let $(V,d)$ be a dg vector space with a homogeneous basis $\{e_a\}$ and associated dual basis $\{t^a\}$. We may view $V$ as a formal graded manifold by considering a formal neighborhood of the origin. Thus a formal graded manifold is naturally pointed with $0$ denoting the distinguished (and only) point. For the structure sheaf we have $\Ocal_V\iso\K\llbracket t^a\rrbracket$ and the tangent sheaf $\Tcal_V$ is generated as an $\Ocal_V$-module by $\{\p_a\}$, where we write $\p_a$ for $\ddt{a}$. The differential $d$ of $V$ corresponds to a degree one vector field $D$, linear in $t$, given by $D=D^b_a t^a \p_b$, where $D^b_a\in\K$ are defined by $d(e_a)=D^b_a e_b$. That $d$ is a differential, i.e. $d^2=0$, is equivalent to $[D,D]=0$. Denoting $sdt^a$ by $\sdt^a$, where $s$ is a formal symbol of degree one, the de Rham algebra is defined by $\Omega^\bullet_V=\odot^\bullet \Tcal^*_V[-1]\iso\K\llbracket t^a,\sdt^s\rrbracket$.

\subsection{Vector field valued differential forms}
\label{vfms}

A \emph{vector field valued differential form}, or \emph{vector form} for short, is a tensor field in  $\Omega^\bullet_V\otimes_{\Ocal_V}\Tcal_V$. We will usually omit the subscript $\Ocal_V$ from the notation. Vector forms where used by Fr\"olicher and Nijenhuis in the study of derivations of $\Omega^\bullet_V$ \cite{Frolicher1956}.

An endomorphism $J\colon\Tcal_V \to \Tcal_V$ can be considered as an element in $\vfmsi{1}$; we identify $J$ given by $J(X^a\p_a)=J^b_a X^a\p_b$, where $X^a, J^b_a\in\Ocal_V$, with $J=J^b_a \sdt^a\p_b$.

Similarly an element $K\in\vfms$ can be viewed as a morphism $K\colon\wedge^\bullet \Tcal_V \to \Tcal_V$; for $K=\sum_i K_i$, with
\[
K_i=K^b_{[a_1\dotsb a_i]}\sdt^{a_1}\dotsb\sdt^{a_i}\p_b\in\vfmsi{i},
\]
and
\[
X=X^{[b_1\dotsb b_j]} \p_{b_1}\wedge\dotsb\wedge\p_{b_j}\in \wedge^j\Tcal_V
\]
we have $K(X)=K^b_{[a_1\dotsb a_i]}X^{[a_1\dotsb a_i]}\p_b$.

\subsection{The Fr\"olicher-Nijenhuis bracket}
\label{FNbracket}

In \cite{Nijenhuis1955} A.~Nijenhuis defined a Lie bracket on vector forms called the Fr\"olicher-Nijenhuis bracket. Let
\[
K=K^{i}_{[a_1\dotsb a_p]}(t)\sdt^{a_1}\dotsb\sdt^{a_p} \p_i
\]
and
\[
L=L^{j}_{[c_1\dotsb c_q]}(t)\sdt^{c_1}\dotsb \sdt^{c_q} \p_j
\]
be vector forms, $p,q\geq 0$. The  Fr\"olicher-Nijenhuis bracket is defined by
\begin{multline*}
\FN{K}{L}=\left(K^{i}_{[a_1\dotsb a_p|}\p_i L^{j}_{|a_{p+1}\dotsb a_{p+q}]}-(-1)^{pq}L^{i}_{[a_1\dotsb a_q|}\p_i K^{j}_{|a_{q+1}\dotsb a_{p+q}]}\right. \\
\left.-p K^{j}_{[a_1\dotsb a_{p-1}|i|}\p_{a_p} L^{i}_{a_{p+1}\dotsb a_{p+q}]}+(-1)^{pq}sL^{j}_{[a_1\dotsb a_{q-1}|i|}\p_{a_q} K^{i}_{a_{q+1}\dotsb a_{p+q}]}\right)\\
\sdt^{a_1}\dotsb\sdt^{a_{p+q}} \p_j.
\end{multline*}

When $p=q=0$ this reduces to the ordinary Lie bracket of vector fields and when $p=0$, $q> 0$ it is called the Lie derivative of $L$ along $K$.

Given two elements of $K,L\in\vfmsi{1}$ we have $\FN{K}{L}\in\vfmsi{2}$. Nijenhuis showed \cite{Nijenhuis1955} that with the above identification between endomorphisms of $\Tcal_V$ and $\vfms$,  we have in fact $\Ncal_J=\FN{J}{J}$. Thus a Nijenhuis structure can equivalently be defined as an element $J\in\vfmsi{1}$ such that $\FN{J}{J}=0$.

\subsection{Nijenhuis structures on graded manifolds}
\label{gradednijenhuis}

When $V$ is a manifold concentrated in degree zero, i.e.~an ordinary manifold, the degree $i$ elements of $\vfms$ are those of $\vfmsi{i}$. We notice that classical Nijenhuis structures are the Maurer-Cartan elements of $\vfms$, i.e.~ degree one elements $J$ satisfying $\FN{J}{J}=0$. 

In the case when $V$ is graded manifold it is natural to retain the definition of Nijenhuis structures on $V$ as Maurer-Cartan elements of $\vfms$. Using the grading induced by the one of $V$, degree one elements of $\vfms$ are now not necessarily in $\vfmsi{1}$. We define a \emph{\Nijinfty\ structure} on a graded manifold $V$ to be a Maurer-Cartan element of $\vfms$. We call a \Nijinfty\ structure on a pointed manifold \emph{pointed} if it satisfies $J|_p=0$, where $p$ is the distinguished point. 

\Nijinfty\ structures where studied by S.~Merkulov in \cite{Merkulov2005}. One of the results therein is that \Nijinfty\ structures correspond to contractible dg manifolds, see Section 5 of \cite{Merkulov2005} for details. A \Nijinfty\ structure can be interpreted as a family of maps $\{J_i\colon\wedge^i\Tcal_V\to \Tcal_V\}_{i\in\N}$ (or as a map $J\colon\pvfs\to \Tcal_V$) satisfying certain quadratic relations. The geometrical (or algebraic) significance of these maps is an interesting and open question.

There is another grading on $\vfms$ given by the tensor power of $\Tcal^*[-1]$; the elements of \emph{weight} $i$ are those of $\vfmsi{i}$.
We define a \emph{graded Nijenhuis structure} to be a Maurer-Cartan element of $\vfms$ of weight one, cf.~graded Lie algebras versus $\ELL_\infty$-algebras.

\subsection{Compatible Nijenhuis structures}

Given two endomorphisms $J$ and $K$ of $\Tcal_V$ one can consider the morphism $\Ncal_{J,K}\colon\Tcal_V\otimes\Tcal_V\to\Tcal_V$ defined by
\begin{equation*}
 \Ncal_{J,K}(X,Y)=JK[X,Y]+[JX,KY]-J[X,KY]-K[JX,Y].
\end{equation*}
When $J$ and $K$ are \emph{commuting}, i.e.~$J\circ K=K\circ J$, we have that $\Ncal_{J,K}$ corresponds to a (1,2) tensor field which in general is not alternating. It was introduced in \cite{Nijenhuis1951} in the study of the problem of when eigenvectors of a tangent bundle endomorphism form an integrable distribution. It was further considered in \cite{Bogoyavlenskij2006} where an explicit description was given of the relation between $\Ncal_J$ and $\Ncal_{A,B}$ when $A$ and $B$ are polynomials in $J$. For arbitrary $J,K\in\vfmsi{1}$ the following identity holds
\begin{equation*}
 \Ncal_{J,K}+\Ncal_{K,J}=\FN{J}{K}.
\end{equation*}

We call two (graded) Nijenhuis structures $J$ and $K$ \emph{compatible} if their sum again is a Nijenhuis structure and we call such a pair a \emph{bi-Nijenhuis structure}. Note that this is equivalent to that $\alpha J + \beta K$ is a Nijenhuis structure for any $\alpha,\beta\in\K$.
 For the Nijenhuis torsion of the sum of two Nijenhuis structures $J$ and $K$ we have
 \[
 \Ncal_{J+K}=\Ncal_J+\Ncal_{J,K}+\Ncal_{K,J}+\Ncal_K=\Ncal_{J,K}+\Ncal_{K,J}=\FN{J}{K}.
 \]
 The compatibility of $J$ and $K$ is thus equivalent to $\FN{J}{K}=0$. Introducing a formal parameter $\hslash$ and considering the linearization in $\hslash$ of the Fr\"olicher-Nijenhuis bracket, the pair $J$ and $K$ is a bi-Nijenhuis structure precisely when
\[
 \FN{J+\hslash K}{J+\hslash K}=0.
\]

\begin{rem*}
 The notion of compatible Nijenhuis structures has been defined differently elsewhere, e.g.~in \cite{Longguang2004} it is defined to be what we call commuting Nijenhuis structures.
\end{rem*}


\section{Operad profiles I: Extracting the operad}
\label{extraction}

We first review the operadic profile of Nijenhuis structures given by S.~Merkulov in \cite{Merkulov2005}. Thereafter we extract the operad encoding the fundamental part of Nijenhuis structures. See \S \ref{resolutionsofoperads} for definitions and notation related to operads. 

\subsection{The operad profile of Nijenhuis structures}

Let $V$ be a dg vector space considered as a formal graded manifold. Recall (\S \ref{gradednijenhuis}) that a pointed graded Nijenhuis structure on $V$ is a degree one element $J\in\vfmsi{1}$ such that $\FN{J}{J}=0$ and $J|_0=0$. Using the notation of \S\ref{fgm} we have
\[
 J=\sum_{i\geq 1} J^b_{(c_1\dotsb c_i) a}t^{c_1}\dotsb t^{c_i}\gamma^a\p_b,
\]
where $J^b_{(c_1\dotsb c_i) a}\in\K$. The vector form $J$ defines a family of degree zero maps $\{j_i\colon  V^{\odot i}\otimes V\to V\}$ by
\[
 j_i(e_{c_1}\odot\dotsb\odot e_{c_i}\otimes e_a)=J^b_{c_1\dotsb c_i a}e_b.
\]
Let $\Jf$ denote the part of $J$ corresponding to $j_1$. It was observed in \cite{Merkulov2005} that $j_1$ gives $V$ a pre-Lie algebra structure if and only if $\FN{\Jf}{\Jf}=0$. We want to encode this in the language of operads. To $j_1$ we associate the corolla $\PLsmall$. Denoting the nontrivial element of $\s_2$ by $(12)$ we depict the elements $\PLsmall$ and $\PLsmall(12)$ by the planar corollas
\[
 \PL{1}{2} \bhs\text{and}\bhs \PL{2}{1},
\]
respectively. The condition $\FN{\Jf}{\Jf}=0$ then translates to
\begin{equation*}
R_{a,b,c}=\PLYPLY{a}{b}{c}-\PLYPLYop{a}{b}{c}- \PLYPLY{a}{c}{b}+\PLYPLYop{a}{c}{b},
\end{equation*}
for $a,b,c$ being the cyclic permutations of $1,2,3$. Let $E_{\text{PL}}=\K\PLsmall\oplus\K\PLsmall(12)=\K[\s_2]$ and $R_{\text{PL}}=R_{1,2,3}\cup R_{2,3,1} \cup R_{3,1,2}$, then the operad of pre-Lie algebras $\Prelie$ is given by $\QO{E_{\text{PL}}}{R_{\text{PL}}}$. This operad was shown to be Koszul in \cite{Chapoton2001} and thus its minimal resolution $\Prelie_\infty$ can be computed explicitly. Representations of $\Prelie_\infty$ correspond to Nijenhuis structures which are linear in $t$. To obtain arbitrary pointed Nijenhuis structures we need to add another fundamental part.

A \emph{homological vector field} on a formal graded manifold $V$ is a degree one vector field $Q$ such that $[Q,Q]=0$. To
\[
 Q=\sum_{i\geq 1} Q^b_{(c_1\dotsb c_i)}t^{c_1}\dotsb t^{c_i}\p_b
\]
we can associate a family of degree one maps $\{q_i\colon\odot^i V\to V\}$, as we did for $J$, by
\[
 q_i(e_{c_1}\odot\dotsb\odot e_{c_i})=Q^b_{(c_1\dotsb c_i)} e_b.
\]
It was observed by Kontsevich \cite{Kontsevich2003} that the condition $[Q,Q]=0$ is equivalent to that $\{q_i\}$ gives $V[-1]$ the structure of an $\ELL_\infty$-algebra. The fundamental operation here is $q_2$ with the other operations being higher homotopies of $q_2$. The corresponding part $\Qf$ of $Q$ is the other fundamental part that is necessary to model Nijenhuis structures. We depict the corolla by which we encode the operation $q_2$ by $\Ysmall$. Since $q_2$ is symmetric the corolla satisfies $\Ysmall(12)=\Ysmall$. The condition $[\Qf,\Qf]=0$ translates to
\begin{equation*}
R_{\text L}=\YY{1}{2}{3}+\YY{2}{3}{1}+\YY{3}{1}{2}.
\end{equation*}
Let $E_{\text L}=\K \Ysmall=\bbone_1[-1]$. Then $\Lieone=\QO{E_{\text L}}{R_{\text L}}$ is the operad of odd Lie algebras and representations of its minimal resolution $\Lieone_\infty$ in a vector space $V$ are $\ELL_\infty$-algebras on $V[-1]$. The reason that $Q$ itself is not visible in classical Nijenhuis structures is that it has no part in degree zero.

The last step in encoding the operad profile of Nijenhuis structures is achieved by considering the interplay between $\Jf$ and $\Qf$. Here we use that the Fr\"olicher-Nijenhuis bracket is an extension of the Lie bracket of vector fields. The compatibility is given by
\[
 \FN{\Jf}{\Qf}=0.
\]
This translates to
\begin{equation*}
 R^{C}_{a,b,c}=\YPLY{a}{b}{c} + \YPLYop{a}{b}{c} + \YPLYop{b}{a}{c} - \PLYYop{a}{b}{c} - \PLYYop{b}{a}{c},
\end{equation*}
for $a,b,c$ being the cyclic permutations of $1,2,3$. Let $R_C=R^{C}_{1,2,3}\cup R^{C}_{2,3,1}\cup R^{C}_{3,1,2}$.

Actually, since $\FN{\Qf}{\Qf}\in\vfmsi{0}$, $\FN{\Jf}{\Qf}\in\vfmsi{1}$, and $\FN{\Jf}{\Jf}\in\vfmsi{2}$, the relations $R_{\text{PL}}$, $R_{\text{L}}$, and $R_{\text{C}}$ can all be simultaneously expressed by the single condition
\[
 \FN{\Qf+\Jf}{\Qf+\Jf}=0.
\]

\begin{defn*}[Merkulov] The operad $\Nijenhuis$ is the quadratic operad 
\[\QO{E_{\text{PL}}\oplus E_{\text L}}{R_{\text{PL}}\cup R_{\text L}\cup R_{\text C}}.
\]
\end{defn*}

The operad $\Nijenhuis$ thus contains the Lie operad and the pre-Lie operad with the operations differing by one in degree and compatible in the sense of $R_C$. See \cite{Merkulov2005} for interpretations of this compatibility.

\begin{rem*}
The operad $\Nijenhuis$ was denoted by $\Preliesquared$ in \cite{Merkulov2005}. We renamed it to avoid confusion with the operad of compatible pre-Lie algebras. 
\end{rem*}

Using the above correspondence between endomorphisms of $V$ and vector forms Merkulov \cite{Merkulov2005} was able to show the following:

\begin{thm}[Merkulov]
\label{Merkulovsthm}
 There is a one-to-one correspondence between representations of $\Omega(\Nijenhuis^{\antishriek})$ in $V$ and pointed Nijenhuis$_\infty$ structures on the formal manifold associated to $V$.
\end{thm}

When $V$ is concentrated in degree zero we obtain precisely classical Nijenhuis structures.

\begin{rem*}
That the \Nijinfty\ structures considered in Theorem \ref{Merkulovsthm} are pointed poses no real problem. Given an arbitrary non-pointed \Nijinfty\ structure on a formal graded manifold $V$, i.e.~an element $J\in\vfms$ such that $\FN{J}{J}=0$ and $J|_0\neq 0$, it can be obtained from $\Nijenhuis_\infty$ by considering representations in $V\oplus\K$. For a formal variable $x$, viewed as a coordinate on $\K$, we have that $xJ\in\Omega^\bullet_{V\oplus\K}\otimes\Tcal_{V\oplus\K}$ vanishes at the distinguished point of $V\oplus \K$ and since $xJ$ still satisfies $\FN{xJ}{xJ}=0$, it corresponds to a representation of $\Nijenhuis_\infty$.
\end{rem*}

\subsection{Extracting the operad of bi-Nijenhuis structures}

Recall that a bi-Nijenhuis structure is pair $(J,K)$ of Nijenhuis structures such that their sum is a Nijenhuis structure. This is equivalent to that the following conditions are satisfied:
\[
 \FN{J}{J}=0,\bhs \FN{J}{K}=0,\bhs \text{and} \bhs \FN{K}{K}=0,
\]
or equivalently 
\[
\FN{J+\hslash K}{J+\hslash K}=0.
\] 
We want to define an operad capturing the fundamental part of this structure, analogously to how a Nijenhuis structure is encoded.

Let $\Jfa$ and $\Jfb$ denote the fundamental parts of $J$ and $K$, respectively, and let $\Qf$ again denote the fundamental part of a homological vector field. That $J$ and $K$ are Nijenhuis structures is encoded by translating
\[
 \FN{\Qf}{\Qf}=0,\mhs \FN{\Qf}{\Jfa}=0,\mhs \FN{\Jfa}{\Jfa}=0, \mhs \FN{\Qf}{\Jfb}=0, \mhs\text{and}\mhs \FN{\Jfb}{\Jfb}=0
\]
to corresponding operadic relations. In doing this we denote the corollas encoding $\Jfa$ and $\Jfb$ by $\PLWsmall$ and $\PLBsmall$ respectively. The compatibility of $J$ and $K$ is captured by
\[
\FN{\Jfa}{\Jfb}=0,
\]
which translates to
\[
R_{a,b,c}=\PLBPLW{a}{b}{c}-\PLBPLWop{a}{b}{c}- \PLBPLW{a}{c}{b}+\PLBPLWop{a}{c}{b}+
\PLWPLB{a}{b}{c}-\PLBPLWop{a}{b}{c}- \PLBPLW{a}{c}{b}+\PLBPLWop{a}{c}{b}.
\]

We make the following definition:
\begin{defn*}
 Let $\circ$ and $\bullet$ be pre-Lie products on a vector space $V$. We call them \emph{compatible} if their sum $\circ+\bullet$, defined by $a(\circ+\bullet)b:=a\circ b + a \bullet b$ is a pre-Lie product. Equivalently, the products are compatible if they satisfy
 \[
 (a\bullet b)\circ c - a\circ (b\bullet c) -(a\bullet c) \circ b + a\circ(c\bullet b)+
 (a\circ b)\bullet c - a\bullet (b\circ c) -(a\circ c) \bullet b + a\bullet(c\circ b)=0.
 \]
 \end{defn*}
Thus the compatibility of $\Jfa$ and $\Jfb$ means that the corresponding maps $j_1$ and $k_1$ give $V$ the structure of compatible pre-Lie algebras. See \cite{Strohmayer2008a} for a treatment of operads encoding compatible structures.

We can encode all the above conditions on $\Jfa$, $\Jfb$, and $\Qf$ by the single equation
\[
 \FN{\Qf+\Jfa+\hslash \Jfb}{\Qf+\Jfa+\hslash \Jfb}=0.
\]
Translated to the language of operads we obtain the following.

\begin{defn*}
We define the quadratic operad $\Binijenhuis$ by $\Binijenhuis=\QO{M}{R}$, where $M=\{M(n)\}_{n\geq 0}$ is the $\smodule$ given by
\[
M(n)=
\begin{cases}
\bbone_2[-1]\oplus\K[\s_2]\oplus\K[\s_2] & \text{ if $n=2$} \\
0 & \text{otherwise}
\end{cases}
\]
and the relations $R$ are given by
\[
\begin{split}
&\PLWPLW{a}{b}{c} - \PLWPLWop{a}{b}{c} - \PLWPLW{a}{c}{b} + \PLWPLWop{a}{c}{b}, \shs \PLBPLB{a}{b}{c} - \PLBPLBop{a}{b}{c} - \PLBPLB{a}{c}{b} + \PLBPLBop{a}{c}{b}, \\
  &\PLWPLB{a}{b}{c} - \PLWPLBop{a}{b}{c} - \PLWPLB{a}{c}{b} + \PLWPLBop{a}{c}{b} + \PLBW{a}{b}{c} - \PLBPLWop{a}{b}{c} - \PLBPLW{a}{c}{b} + \PLBPLWop{a}{c}{b},\\
  &\YY{1}{2}{3}+\YY{2}{3}{1}+\YY{3}{1}{2}, \shs  \YPLW{a}{b}{c} + \YPLWop{a}{b}{c} + \YPLWop{b}{a}{c} - \PLWYop{a}{b}{c} - \PLWYop{b}{a}{c}, \\
  &\YPLB{a}{b}{c} + \YPLBop{a}{b}{c} + \YPLBop{b}{a}{c} - \PLBYop{a}{b}{c} - \PLBYop{b}{a}{c}.
\end{split}
\]
\end{defn*}

To sum up this definition, a $\Binijenhuis$ algebra is a pair of $\Nijenhuis$ algebras sharing the same Lie bracket and such that the pre-Lie products are compatible.


\section{Operad profiles II: Computing the resolution}
\label{resolution}

We first recall some notions related to resolutions of operads, then we give a brief account of Hoffbeck's PBW bases for operads. Next, we use this method to prove the Koszulness of the operads of Nijenhuis structures and bi-Nijenhuis structures. Finally we give an explicit description of the minimal resolution of the bi-Nijenhuis operad.

\subsection{Resolutions of operads}
\label{resolutionsofoperads}

The free operad $\Free(M)$ on a dg $\s$-module $(M,d)$ has a natural differential induced by $d$. It is given by the alternating sum of applying $d$ to the decorated vertices of an element of $\Free(M)$. We call a dg operad $(\Free(M),\delta)$ on a dg $\s$-module $M$ \emph{quasi-free} if for the differential $\delta=d+d'$ we have that $d$ is the differential induced by the one of $M$ and $d'$ is an arbitrary derivation of $\Free(M)$. If for some operad $\Pcal$ a morphism $\phi\colon(\Free(M),\delta)\to\Pcal$ induces an isomorphism on cohomology we say that $(\Free(M),\delta,\phi)$ is a \emph{quasi-free resolution} of $\Pcal$. We call it \emph{minimal} if it satisfies $d'(M)\subset\Free_{(2)}(M)$. Here $\Free_{(i)}(M)$ denotes the $\s$-submodule of $\Free(M)$ consisting of the elements whose underlying trees have $i$ vertices. A quasi-cofree cooperad is defined analogously.

The \emph{suspension} $\Sigma\Pcal$ of an operad $\Pcal$ is defined by $\Sigma\Pcal(n)=\Pcal(n)[-1]$, similarly the \emph{desuspension} $\Sigma^{-1}\Pcal$ is defined by $\Sigma^{-1}\Pcal(n)=\Pcal(n)[1]$. An operad is called \emph{augmented} if there is a morphism $\epsilon\colon\Pcal\to I$, where $I$ is the operad defined by $I(1)=\K$ and zero otherwise. The \emph{augmentation ideal} is defined by $\ol{\Pcal}=\ker \epsilon$. Similarly we define the coaugmentation of a cooperad $\Ccal$ as $\ol{\Ccal}=\coker \eta$, for $\eta\colon I \to \Ccal$.

The \emph{Bar construction} $\BC(\Pcal)$ of an augmented dg operad $(\Pcal,d)$ is the quasi-cofree cooperad $(\Free^{c}(\Sigma^{-1}\ol{\Pcal}),\delta)$, where $\delta=d+d'$ and $d'$ is induced by the composition product in $\Pcal$. The \emph{Koszul dual} of $\Pcal$ is the cooperad $\Pcal^{\antishriek}_{(s)}=\Ho_s(\BC_{(\bullet)}(\Pcal)_{(s)})$. Here the grading in parentheses is a weight grading induced from the natural weight grading in the free operad given by the number of vertices, see \cite{Strohmayer2008} for details. We say that $\Pcal$ is \emph{Koszul} if the canonical inclusion $\Pcal^{\antishriek}\hookrightarrow\BC{\Pcal}$ is a quasi-isomorphism.

The \emph{Cobar construction} $\CBC(\Ccal)$ of a dg cooperad $(\Ccal,d)$ is the quasi-free operad $(\Free(\Sigma\ol{\Ccal}),\delta)$, where $\delta=d+d'$ and $d'$ is induced by the cocomposition product in $\Ccal$. In \cite{Ginzburg1994} it was shown that if $\Pcal$ is Koszul, then $\CBC(\Pcal^{\antishriek})$ is a quasi-free resolution of $\Pcal$.

When $\Pcal=\QO{M}{R}$ is a quadratic operad (i.e.~$R\subset\Free_{(2)}(M)$) which is finite dimensional in each arity $\Pcal(n)$ there is an explicit way to calculate the part $d'$ of the differential of $\CBC(\Pcal^{\antishriek})$. Let $M^{\cz}$ denote the \emph{Czech dual} $\s$-module defined by $M^{\cz}(n)=M(n)^*\otimes\sgn_n$. There is a natural pairing $\langle\_,\_\rangle\colon\Free_{(2)}(M^{\cz})\otimes\Free_{(2)}(M)\to\K$, defined to be zero unless the underlying trees and decorations match, again see \cite{Strohmayer2008} for details. The orthogonal complement to $R$ under this pairing is denoted by $R^{\perp}$. We define the \emph{Koszul dual operad} of $\Pcal$ to be $\Pcal^{!}=\QO{M^{\cz}}{R^{\perp}}$. It is related to the Koszul dual by
\begin{equation}
\label{corkoszuldualkoszuldualoperad}
(\Pcal^{\antishriek})_{(s)}(n)\iso\Sigma^{-s}((\Pcal^!)_{(s)}(n))^{\cz}.
\end{equation}
where the isomorphism is of cooperads and the cocomposition product on the right hand side is the linear dual of the composition product of $\Pcal^{!}$. Thus, computing the Koszul dual operad and its composition product gives us a tractable way of determining the differential of the bar construction on $\Pcal^{\antishriek}$.

\subsection{A user's guide to PBW-bases for operads}
\label{PBWparagraph}

This section is a summary of the results we need from \cite{Hoffbeck2008}.

By a \emph{tree} we mean a directed rooted tree with an additional set of edges, called external, attached on one side only to the leaves of the tree. The external edges are labeled by the numbers $1,\dots,n_e$, where $n_e$ denotes the number of external edges (for esthetical and suggestive reasons we depict the trees as having an additional external edge attached to the root vertex). The other edges are called internal. An input edge of a vertex $v$ is an internal edge attached to and directed towards $v$ or an external edge attached to $v$. Let $E_v$ denote the set of input edges of $v$ and let $n_v$ denote the cardinality of $E_v$. A tree decorated with an $\s$-module $M$ is a tree together with an assignment to each vertex $v$ of an element of $M(n_v)$. The free operad on an $\s$-module is spanned as a $\K$-module by all isomorphism classes of decorated trees. The action of $\s$ is given by permutation of the labels of the external edges.

A tree is called \emph{reduced} if all vertices have at least one incoming edge. If an $\s$-module satisfies $M(0)=0$, then all non-zero decorated trees in $\Free(M)$ are reduced. An operad is called reduced if it is spanned by reduced trees. A reduced tree has a natural planar representation.
\begin{enumerate}
\item To every $e\in E_v$ we associate the minimum of the labels of the external edges of the tree which are linked to $e$ (we consider an external edge to be linked to itself).
\item We place the edges of $E_v$ (and thus the vertices directly above) from left to right above $v$ in ascending order.
\end{enumerate}

Let $M$ be an $\s$-module, let $B^M$ be a $\K$-basis of $M$, and let $B^{\Free(M)}$ denote all trees decorated with elements of $B^M$. The set $B^{\Free(M)}$ is a $\K$-basis of $\Free(M)$. For a tree $\tau$, we denote by $B^{\Free(M)}_\tau$ the subset of $B^{\Free(M)}$ consisting of the elements whose underlying tree is $\tau$.

Given an order of the elements of $B^M$, and using the above defined planar representations of trees, we define an order on $B^{\Free(M)(n)}$. To each decorated tree in $\Free(M)(n)$ we associate a sequence of $n$ words in the alphabet $B^M$ as follows. There is a unique path of vertices from the root of the tree to the external edge labeled by $i$. Let $a_i$ be the word consisting (from left to right) of the labels of these vertices (from bottom to top). Thus we obtain the sequence $\bar{a}=(a_1,\dotsc, a_n)$. The words are ordered by the length lexicographical order; we first compare two words $a$ and $b$ by their length ($a < b$ if $l(a)<l(b)$, where $l$ is the length) and if the lengths are equal we compare them lexicographically using the order on $B^M$. We compare two sequences $\bar{a}$ and $\bar{b}$ associated to $\alpha,\beta\in B^{\Free(M)}(n)$ by first comparing $a_1$ with $b_1$, next $a_2$ with $b_2$, and so forth. This order is compatible with the operad structure of $\Free(M)$, see \cite{Hoffbeck2008} for details.

To each internal edge $e$ of a tree $\tau$ we define the restricted tree $\tau_e$ as follows. The vertices of $\tau_e$ are the two vertices $v_1,v_2$ adjacent to $e$. The edges of $\tau_e$ are all edges adjacent to $v_1$ and $v_2$. The external edges of $\tau_e$ are given labels according to which labels are directly linked to them; the external edge of $\tau_e$ linked to the external edge of $\tau$ labeled by $1$ is given this label, the external edge of $\tau_e$ linked to the external edge of $\tau$ which has the least of the labels not linked to the previous edge is given the label $2$, and so forth.

\begin{defn*}
 A \emph{PBW-basis} for a quadratic operad $\Pcal=\QO{M}{R}$ is a basis $B^\Pcal\subset B^{\Free(M)}$ of $\Pcal$ satisfying the conditions:
\begin{enumerate}
 \item \label{conditionzero}
 $\bbone \in B^{\Pcal}$ and $B^M\subset B^{\Pcal}$,
 \item \label{conditionone}
 for $\alpha\in B^{\Free(M)}$, either $\alpha\in B^\Pcal$ or the elements of the basis $\gamma_i\in B^{\Pcal}$ which appear in the unique decomposition $\alpha\equiv\sum_i c^i \gamma_i$, satisfy $\gamma_i>\alpha$,
 \item \label{conditiontwo}a decorated tree $\alpha\in B^{\Free(M)}_\tau$ is in $B^{\Pcal}_\tau$ if and only if for every internal edge $e$ of $\tau$, the restricted decorated tree $\alpha|_{\tau_e}$ is in $B^{\Pcal}_{\tau_e}$.
\end{enumerate}
\end{defn*}

The following result makes it easier to verify that a given basis is a PBW-basis.

\begin{prop}[Hoffbeck]
\label{twovertexcondition}
Let $M$ be finitely generated. If condition \eqref{conditionone} is verified when the underlying tree of $\alpha$ has two vertices and conditions \eqref{conditionzero} and \eqref{conditiontwo} are satisfied, then condition \eqref{conditionone} is satisfied for all $\alpha$.
\end{prop}

Using a filtration induced by the ordering on $\Free(M)$ Hoffbeck was able to prove the following theorem.

\begin{thm}[Hoffbeck]
A reduced operad which has a PBW basis is Koszul.
\end{thm}

\subsection{Koszulness of the operad of Nijenhuis structures}

The Koszul dual operad of $\Nijenhuis$ is described in \cite{Merkulov2005}. It has the operads $\Perm$, of permutative algebras (see e.g.~\cite{Chapoton2001}), and $\Com$ as suboperads. Let $\PLsmall$ and $\PLopsmall=\PLsmall(12)$ denote the generators of $\Perm$ and let $\Ysmall$ denote the generator of the $\Com$.

\begin{rem*}
We use the notation
\[
\text{$\PLop{1}{2}$ instead of $\PL{2}{1}$}
\]
in order to write the trees using the planar representation of \S \ref{PBWparagraph}.
\end{rem*}

\begin{thm}
 The following is a PBW-basis of $\Nijenhuis^{!}$ with respect to the ordering $\PLopsmall < \Ysmall < \PLsmall$
\[
B^{\Nijenhuis}=
\left\{
\nijPBWa
\right\}_{\substack{i_1<\dotsb<i_r \\ j_1< \dotsb <j_s}}
\hspace{-5pt}\cup\hspace{15pt}
\left\{
\nijPBWb
\right\}_{\substack{i_1<\dotsb<i_r \\ j_1< \dotsb <j_s}}.
\]
\end{thm}

\begin{proof}
Through straightforward graph calculations one can verify that $B^{\Nijenhuis^{!}}$ is a basis of $\Nijenhuis^{!}$, cf.~\S 3.6 of \cite{Merkulov2005} and note that the dotted edges here correspond to the wavy edges in \cite{Merkulov2005}. Condition \eqref{conditionzero} is obviously satisfied and condition \eqref{conditiontwo} is easily verified. We denote the elements of $B^{\Nijenhuis^{!}(3)}$ by
\[
\begin{split}
\gamma_1=\YPLY{1}{2}{3},\shs \gamma_2=\YPLY{1}{3}{2},\shs \gamma_3=\PLYtwY{1}{2}{3},\shs \gamma_{4}=\YY{1}{2}{3}\\
\gamma_5=\PLYPLY{1}{2}{3},\shs \gamma_6=\PLYtwPLY{1}{2}{3},\shs \gamma_7=\PLYtwPLY{1}{3}{2}.
\end{split}
\]
To show that \eqref{conditionone} is satisfied it is sufficient, using Proposition \ref{twovertexcondition}, to observe that for any decorated two-vertex graph $\alpha\in B^{\Free(M)}\setminus B^{\Nijenhuis^{!}}$ with $\alpha=\sum c_i \gamma_i$, $c_i\in\K$, we have $c_i\neq 0\implies\gamma_i > \alpha$. Here $M=\K\PLsmall\oplus\K\PLopsmall\oplus\K\Ysmall$.
\end{proof}

\begin{cor}
The operad $\Nijenhuis^{!}$, and thus also $\Nijenhuis$, is Koszul.
\end{cor}

Together with Theorem \ref{Merkulovsthm} we obtain.

\begin{cor}
 There is a one-to-one correspondence between representations of $\Nijenhuis_\infty$ in $V$ and pointed \Nijinfty\ structures on the formal graded manifold associated to $V$.
\end{cor}

Thus the differential given in \S 3.6 of \cite{Merkulov2005} is really the differential of the minimal resolution of $\Nijenhuis$. We note that Theorem \ref{nongradednijenhuisthm} is the special case when $V$ equals $\R^n$.

\subsection{The Koszul dual operad of $\Binijenhuis$}

\begin{prop}
The Koszul dual operad of $\Binijenhuis$ is the quadratic operad
\[
\Binijenhuis^{!}=\QO{N}{S}.
\]
Here $N=\{N(n)\}_{n\geq 0}$ is the $\smodule$ given by
\[
N(n)=
\begin{cases}
\sgn_2[1]\oplus\K[\s_2]\otimes\sgn_2\oplus\K[\s_2]\otimes\sgn_2 & \text{ if $n=2$} \\
0 & \text{otherwise}.
\end{cases}
\]
If we denote the natural generator of $\sgn_2[-1]$ by $\Ysmall$ and the generators of $\K[\s_2]\otimes\sgn_2\oplus\K[\s_2]\otimes\sgn_2$ by $\PLWsmall$ and $\PLBsmall$ the relations $S$ are given by
\[
\begin{split}
 &\PLWPLW{a}{b}{c} - \PLWPLWop{a}{b}{c}, \shs \PLWPLW{a}{b}{c} - \PLWPLW{a}{c}{b}, \shs \PLBPLB{a}{b}{c} - \PLBPLBop{a}{b}{c}, \shs \PLBPLB{a}{b}{c} - \PLBPLB{a}{c}{b},\\
 &\PLWPLB{a}{b}{c} - \PLWPLBop{a}{b}{c}, \shs \PLWPLB{a}{b}{c} - \PLWPLB{a}{c}{b}, \shs \PLBPLW{a}{b}{c} - \PLBPLWop{a}{b}{c},\shs \PLBPLW{a}{b}{c} - \PLBPLW{a}{c}{b}, \\
 &\PLWPLB{a}{b}{c} - \PLBPLW{a}{b}{c}, \shs \PLWPLBop{a}{b}{c} - \PLBPLWop{a}{b}{c}, \shs \YY{a}{b}{c} - \YYop{a}{b}{c},\shs
 \PLWY{a}{b}{c}-\YPLW{a}{c}{b},\\
 & \YPLWop{a}{b}{c}-\YPLW{a}{c}{b}+\YPLW{a}{b}{c},\shs \PLBY{a}{b}{c}-\YPLB{a}{c}{b}, \shs \YPLBop{a}{b}{c}-\YPLB{a}{c}{b}+\YPLB{a}{b}{c}.
\end{split}
\]
\end{prop}
\begin{proof}
For $\Binijenhuis=\QO{M}{R}$ we first note that $N=M^{\cz}$. Under the pairing of \S\ref{resolutionsofoperads} it is easy to verify that $(S)$ indeed is the orthogonal complement to $(R)$.
\end{proof}

For $H$ a subgroup of $G$ and $M$ and $H$-module we define $\Ind^G_H M:=\K[G]\otimes_{\K[H]} M$. Through straightforward graph calculations we obtain the following result.

\begin{prop}
\label{koszuldualbasis}
The underlying $\s$-module of $\Binijenhuis^{!}$ is given by
\begin{multline*}
\Binijenhuis^{!}(n)=\\
\begin{cases}
\bigoplus_{0\leq p \leq n-1}(\Ind^{\s_{n}}_{\s_{n-p}\times\s_{p}}\sgn_{n-p}\otimes
(\underbrace{\bbone_{p}\oplus\dotsb\oplus\bbone_{p}}_{\text{$p+1$ terms}})[n-1-p])   & \text{if $n\geq2$}\\
0 & \text{otherwise}.
\end{cases}
\end{multline*}
Explicitly, a $\K$-basis for $\Binijenhuis^{!}(n)$ is given by
\begin{equation}
\label{nijenhuiskoszuldualoperadbasis}
\left\{ \shs \YWBcorolla \shs \right\}_{\substack{i\geq 1, j\geq 0 \\  (i,n-i)\text{-shuffles} \shs \sigma } }.
\end{equation}
\end{prop}

\subsection{Koszulness of $\Binijenhuis$}

Through a slight modification of the basis \eqref{nijenhuiskoszuldualoperadbasis} we obtain a PBW-basis for $\Binijenhuis^{!}$.

\begin{thm}
 The following is a PBW-basis of $\Binijenhuis^{!}$ with respect to the ordering $\PLWopsmall < \PLBopsmall < \Ysmall < \PLWsmall < \PLBsmall$
\begin{multline*}
B^{\Binijenhuis^{!}}=\\
\left\{
\binijPBWa
\right\}_{\substack{i_1<\dotsb<i_r \\ j_1< \dotsb <j_{s+t}}}
\hspace{-25pt}\cup\hspace{18pt}
\left\{
\binijPBWb
\right\}_{\substack{i_1<\dotsb<i_r \\ j_1< \dotsb <j_{s+t}}}
\hspace{-31pt}\cup\hspace{10pt}
\left\{
\binijPBWc
\right\}_{\substack{i_1<\dotsb<i_r \\ j_1< \dotsb <j_s}}\hspace{-10pt}.
\end{multline*}
\end{thm}

\begin{proof}
Through straightforward graph calculations one can verify that $B^{\Binijenhuis^{!}}$ is a basis of $\Binijenhuis^{!}$, cf.~the basis \eqref{nijenhuiskoszuldualoperadbasis}. Condition \eqref{conditionzero} is obviously satisfied and condition \eqref{conditiontwo} is easily verified. We denote the elements of $B^{\Binijenhuis^{!}(3)}$ by

\begin{align*}
&\gamma_1=\YPLW{1}{2}{3},\abshs \gamma_2=\YPLW{1}{3}{2},\abshs \gamma_3=\PLWtwY{1}{2}{3}, \abshs
\gamma_4=\YPLB{1}{2}{3},\abshs \gamma_5=\YPLB{1}{3}{2},\abshs \gamma_6=\PLBtwY{1}{2}{3},\\
&\gamma_7=\PLWPLW{1}{2}{3},\abshs \gamma_8=\PLWtwPLW{1}{2}{3},\abshs \gamma_9=\PLWtwPLW{1}{3}{2},\abshs
\gamma_{10}=\PLBPLB{1}{2}{3},\abshs \gamma_{11}=\PLBtwPLB{1}{2}{3},\abshs \gamma_{12}=\PLBtwPLB{1}{3}{2},\\
&\gamma_{13}=\YY{1}{2}{3}.
\end{align*}

To show that \eqref{conditionone} is satisfied it is sufficient, using Proposition \ref{twovertexcondition}, to observe that for any decorated two-vertex graph $\alpha\in B^{\Free(N)}\setminus B^{\Binijenhuis^{!}}$ with $\alpha=\sum c_i \gamma_i$, $c_i\in\K$, we have $c_i\neq 0\implies\gamma_i > \alpha$. Here $N=\K\PLWsmall\oplus\K\PLWopsmall\oplus\K\PLBsmall\oplus\K\PLBopsmall\oplus\K\Ysmall$.
\end{proof}

\begin{cor}
The operad $\Binijenhuis^{!}$, and thus also $\Binijenhuis$, is Koszul.
\end{cor}

\subsection{The minimal resolution of $\Binijenhuis$}

\begin{thm}
\label{binijenhuisdifferential}
The minimal resolution  $\Binijenhuis_\infty$ of the operad $\Binijenhuis$ is the quasi-free operad $(\Free(E),\delta)$ where $E=\{E(n)\}_{n\geq 1}$ is the $\smodule$
\[
E(n)=
\begin{cases}
\bigoplus_{0\leq j \leq n-1}(\Ind^{\s_{n}}_{\s_{n-j}\times\s_{j}}\bbone_{n-j}\otimes
(\underbrace{\sgn_{j}\oplus\dotsb\oplus\sgn_{j}}_{\text{$j+1$ terms}})[j-1])   & \text{if $n\geq2$}\\
0 & \text{otherwise}.
\end{cases}
\]
We denote the basis element of $E(n)$ corresponding to the basis element of $\Binijenhuis^!(n)$ with $k$ white operations by
\[
\nijenhuisboxcorolla{I}{J}{k}\sim \YWBcorollab\shs,
\]
where $I=\{\sigma(1),\dotsc,\sigma(i)\}$ and $J=\{\sigma(i+1),\dotsc,\sigma(n)\}$. Note that it is symmetric in the input legs labeled by $I$, skew-symmetric in the ones labeled by $J$, and of degree $1-|J|$. The differential of $\Binijenhuis_\infty$ is then given by

\begin{multline*}
\delta\colon\nijenhuisboxcorolla{I}{J}{k} \mapsto \\\sum_{\substack{k_1+k_2=k \\ I_1 \sqcup I_2 =I \\ I_1 \sqcup I_2 \\ J_1 \sqcup \{J_2\} \sqcup J_3}}
\hspace{-5pt}(-1)^{\epsilon_1}\hspace{5pt}
\nijenhuissplitboxcorollaA
-\sum_{\substack{0 \leq j \leq n-1 \\ 1 \leq k \leq m \\ i_1+i_2=i \\
(j,n-j)\text{-shuffles} \shs \sigma \\ (k,m-k)\text{-shuffles} \shs \tau }}
\hspace{-15pt}(-1)^{\epsilon_2}\hspace{5pt}
\nijenhuissplitboxcorollaB \shs.
\end{multline*}
Here $\epsilon_1=|J_1|+|J_3|+\sigma(J_1\sqcup J_2 \sqcup J_3)$ and $\epsilon_1=|J_2|+\sigma(J_1\sqcup J_2)$, and $\sigma(J_1\sqcup J_2 \sqcup J_3)$ and $\sigma(J_1\sqcup J_2)$ denote the parities of the permutations $J\mapsto J_1\sqcup J_2 \sqcup J_3$ and $J\mapsto J_1\sqcup J_2$, where we assume the elements of each disjoint set to be ordered ascendingly.

\end{thm}

\begin{proof}
From the Koszulness of $\Binijenhuis$ it follows that $\Binijenhuis_\infty=\Omega(\Binijenhuis^{\antishriek})$ is a quasi-free resolution of $\Binijenhuis$. The cobar construction is given by $\Omega(\Binijenhuis^{\antishriek})=\Free(\Sigma\ol{\Binijenhuis^{\antishriek}})$. Since $\Binijenhuis^{\antishriek}(n)$ is concentrated in weight $n-1$ it follows from \eqref{corkoszuldualkoszuldualoperad} and Proposition \ref{koszuldualbasis} that
\[
\Binijenhuis^{\antishriek}(n)=\bigoplus_{0\leq j \leq n-1}(\Ind^{\s_{n}}_{\s_{n-j}\times\s_{j}}\bbone_{n-j}\otimes
(\underbrace{\sgn_{j}\oplus\dotsb\oplus\sgn_{j}}_{\text{$j+1$ terms}})[j]).
\]
Setting $E=\Sigma\ol{\Binijenhuis^{\antishriek}}$ the first assertion of the theorem follows. Since $\Binijenhuis$ has zero differential, it follows that the differential $\delta$ of $\Omega(\Binijenhuis^{\antishriek})$ is fully determined by the cocomposition coproduct of $\Binijenhuis^{\antishriek}$. Through tedious but straightforward graph calculations one can determine the composition product of $\Binijenhuis^{!}$. Considering the linear dual of this product yields the differential $\delta$.
\end{proof}


\section{Geometrical interpretation of $\Binijenhuis_\infty$}
\label{geominterpret}

In this section we first show how representations of $\Binijenhuis_\infty$ correspond to pointed \Binijinfty\ structures. Then we present a conceptual framework in which we formulate this correspondence. Finally we observe that we recover ordinary bi-Nijenhuis structures when considering manifolds concentrated in degree zero.

\subsection{Representations of $\Binijenhuis_\infty$ as vector forms}
\label{correspondence}

Let $(V,d)$ be a dg vector space and let $\End_V$ denote the endomorphism operad of $V$. Since a morphism from a free operad is uniquely determined by the image of its generators, a morphism of graded operads $\rho\colon\Binijenhuis_\infty\to\End_V$ is equivalent to a family of degree zero linear maps
\[
 \{ \leftsub{k}{\mu}_{i,j}\colon V^{\odot i}\otimes V^{\wedge j} \to V[1-j]\}_{\substack{j \geq 0 ,i\geq 1\\ i+j \geq 2 \\ 0 \leq k \leq j}}.
\]
Such a morphism $\rho$ is a representation of $\Binijenhuis_\infty$ in $(V,d)$ if and only if it satisfies $\rho \delta|_{E}=d\rho|_E$, where $d$ denotes also the differential induced by $d$ on $\End_V$ and $E$ the generators of $\Binijenhuis_\infty$. This translates to a family of quadratic relations among the maps $\leftsub{k}{\mu}_{i,j}$.

From such a family of maps we construct a family of vector forms as follows. Let
\[
{\leftsub{k}{\Gamma}_{i,j}}=\frac{1}{i!j!}\leftsub{k}{\Gamma}^{c}_{(a_1 \dotsb a_i)[ b_1\dotsb b_j]}t^{a_1}\dotsb t^{a_i} \sdt^{b_1}\dotsb\sdt^{b_j}\otimes\p_c,
\]
where the numbers $\leftsub{k}{\Gamma}^{l}_{(a_1 \dotsb a_i)[ b_1\dotsb b_j]}\in\K$ are given by
\[
\begin{cases}
 \leftsub{k}{\mu}_{i,j}(e_{a_1}\odot \dotsb\odot e_{a_i} \otimes e_{b_1}\wedge \dotsb \wedge e_{b_j})=\leftsub{k}{\Gamma}^{c}_{(a_1 \dotsb a_i)[ b_1\dotsb b_j]} e_c & \text{if $i+j \geq 2$} \\
  d(e_{a_1})=-\leftsub{0}{\Gamma}^{c}_{a_1}e_{c}  & \text{if $i=1$ and $j=0$}.
\end{cases}
 \]
We assemble these vector forms into a power series in the formal parameter $\hslash$ with vector form coefficients, i.e.~into an element $\Gamma\in\vfmsh$,
\[
 \Gamma=\sum_{k \geq 0}\leftsub{k}{\Gamma}\hslash^k,\bhs\bhs \text{where $\leftsub{k}{\Gamma}=\sum_{\substack{i\geq 1 \\ j \geq k}}{\leftsub{k}{\Gamma}_{i,j}}$}.
\]
We introduce $\hslash$ to distinguish vector forms of the same weight but corresponding to different maps. We write $\FNh{\_}{\_}$ for the linearization in $\hslash$ of the Fr\"olicher-Nijenhuis bracket. Note that $\leftsub{k}{\Gamma}$ comes from the part of the representation of $\Binijenhuis_\infty$ which is obtained from the basis elements with $k$ white operations. Further we observe that $\leftsub{k}{\Gamma}\in\vfmshi{\geq k}$ and that the elements with this property form a Lie subalgebra $\gfrak_V$ of $(\vfmsh,\FNh{\_}{\_})$.

\begin{thm}
\label{main-thm}
A family of maps
\[
 \{ \leftsub{k}{\mu}_{i,j}\colon V^{\odot i}\otimes V^{\wedge j} \to V[1-j]\}_{\substack{j \geq 0 ,i\geq 1\\ i+j \geq 2 \\ 0 \leq k \leq j}}
\]
 is a representation of $\Binijenhuis_\infty$ in $(V,d)$ if and only if the corresponding power series $\Gamma\in\gfrak_V$ satisfies the conditions
 \begin{enumerate}
 \item \label{degreeone} $|\Gamma|=1$,
 \item \label{mc} $\FNh{\Gamma}{\Gamma}=0$,
 \item \label{pointed} $\Gamma|_{0}=0$.
 \end{enumerate}
\end{thm}

\begin{proof}
Under the correspondence between vector form power series and families of maps described above, $\rho\sim\Gamma$, it is straightforward to show that $d \rho|_E= \rho \delta|_E$ is equivalent to $\FNh{\Gamma}{\Gamma}=0$. Conditions \eqref{degreeone} and \eqref{pointed} corresponds to that the degrees of the maps $\leftsub{k}{\mu}_{i,j}$ are $1-j$ and that $i\geq 1$, respectively.
\end{proof}

Note that Theorem \ref{introthmgraded} is just another formulation of the preceding theorem. In analogy with \Nijinfty\ structures we make the following definition.

\begin{defn*}
 A \Binijinfty\ structure on a manifold $V$ is a Maurer-Cartan element in $\gfrak_V$.
\end{defn*}

\subsection{A conceptual framework}

The previous paragraph can be rephrased in a nice way using two well-known results about operads.

Let $(\Ccal,\Delta,\delta_\Ccal)$ be a dg cooperad and $(\Pcal,\mu,\delta_\Pcal)$ a dg operad. The collection $\Pcal^\Ccal=\Hom(\Ccal,\Pcal)$ of all graded $\K$-module homomorphisms is an $\s$-module with components $\Hom(\Ccal,\Pcal)(n)=\Hom(\Ccal(n),\Pcal(n))$ and the action given by $(f\sigma)(x)=f(x\sigma^{-1})\sigma$. The invariants of this action are the $\s$-equivariant maps. The $\s$-module $\Pcal^\Ccal$ has an operad structure \cite{Berger2003} defined as follows. For a two-level tree $\tau$ let $\mu_\tau$ and $\leftsub{\tau}{\Delta}$ be the restrictions of $\mu$ and $\Delta$ to the parts of $\Pcal\circ\Pcal$ and $\Ccal\circ\Ccal$ given by $\tau$ and let $\tau(f;f_1,\dots,f_k)\colon\Ccal\circ\Ccal|_\tau\to\Pcal\circ\Pcal|_\tau$ be defined by applying the morphisms $f,f_1,\dotsc,f_k$ to the decorations of the appropriate vertices. The composition product in $\Pcal^\Ccal$ is then given by
\[
 \mu(f;f_1,\dotsc,f_k):=\mu_\tau\circ \tau(f;f_1,\dotsc,f_k)\circ \leftsub{\tau}{\Delta}.
\]
The differential $\delta_\Pcal$ and codifferential $\delta_\Ccal$ induce a differential $\partial$ given by $\partial(f)=\delta_\Pcal\circ f - (-1)^{|f|} f\circ\delta_\Ccal$. This gives $\Pcal^\Ccal$ a dg operad structure which descends to the invariants.

The total space $\Pcal^{\tot}=\oplus_n\Pcal(n)$ of a dg operad $\Pcal$ has the structure of a pre-Lie algebra \cite{Kapranov2001} given by $p\circ q =\sum_i p\circ_i q$, where $i$ ranges over all possible values. Thus it is a Lie algebra with $[p,q]=p\circ q - (-1)^{|p||q|} q \circ p$. Together with the differential of $\Pcal$ this makes $\Pcal^{\tot}$ a dg Lie algebra.

Let $\Qcal$ be a Koszul operad with zero differential, let $\Ccal$ denote the cooperad $\Sigma\ol{\Qcal^{\antishriek}}$, let $(V,d)$ be dg vector space, and denote $\End_V$ by $\Pcal$. The dg operad morphisms $\Qcal_\infty\to\Pcal$ correspond to the $\s$-invariants $\rho\in\Pcal^{\Ccal}$ such that $\rho\delta=d\rho$. When $\Qcal$ has zero differential, the differential of $\Qcal$ is completely determined by the cocomposition coproduct of $\Ccal$. That an element $\rho\in\Pcal^{\Ccal}$  satisfies $\rho\delta=d\rho$ is thus equivalent to that it satisfies $\partial(\rho)+\half[\rho,\rho]=0$. In other words, a representation of $\Binijenhuis_\infty$ in $(V,d)$ is a Maurer-Cartan element in the Lie algebra $\Lcal_\Qcal(V):=(((\Pcal^{\Ccal})^{\s})^{\tot},[\_,\_],\partial)$.

Let $\tilde{\gfrak}_V$ denote the Lie subalgebra of $\gfrak_V$ given by $\{\Gamma\in\gfrak_V;\text{$\Gamma|_0=0$ and $\leftsub{0}{\Gamma}_{1,0}=0$}\}$. Recall (\S \ref{fgm}) that the differential of a dg vector space $(V,d)$ induces a vector field $D\in\Tcal_V$. This vector field in turn induces a differential $\delta_D$ on $\gfrak_V$ given by $\FNh{D}{\_}$. If we write $\tilde{\Gamma}=\Gamma-D$ for an element $\Gamma\in\gfrak_V$, then $\FNh{\Gamma}{\Gamma}=0$ is equivalent to  $\delta_D(\tilde{\Gamma})+\half\FNh{\tilde{\Gamma}}{\tilde{\Gamma}}=0$, i.e.~a bi-Nijenhuis$_{\infty}$ structure on $V$ is a Maurer-Cartan element in $(\tilde{\gfrak}_V,\FNh{\_}{\_},\delta_D)$.

We define $\Phi\colon\Lcal_{\Binijenhuis}(V)\to\tilde{\gfrak}_V$ to be the vector space morphism given by $\Phi(\rho)=\tilde{\Gamma}$, where $\Gamma$ is the power series given by the correspondence in the previous section. Theorem \ref{main-thm} is now a corollary of the following observation.

\begin{prop}
The morphism $\Phi$ is an isomorphism of dg Lie algebras.
\end{prop}

\subsection{Representations of $\Binijenhuis_\infty$ in non-graded vector spaces}

If the vector space $V$ is concentrated in degree zero, then a representation of $\Binijenhuis_\infty$ in $V$ corresponds to an element $\Gamma=\leftsub{0}{\Gamma}+\leftsub{1}{\Gamma}\hslash\in\tilde{\gfrak}_V$ such that $\leftsub{k}{\Gamma}\in\vfmsi{1}$ and $\leftsub{k}{\Gamma}|_0=0$, i.e.~to a pointed bi-Nijenhuis structure on the formal manifold associated to $V$. In particular this proves Theorem \ref{introthmnongraded}.

\bibliographystyle{alpha}
\bibliography{binijenhuis}

\end{document}